\documentclass[11pt,a4paper]{article}
\usepackage{amsfonts}
\usepackage{amsmath}
\usepackage{amssymb}
\usepackage{amsthm}

\usepackage{enumerate}

\usepackage[top=1in,bottom=1in,left=1in,right=1in]{geometry}
\usepackage{color}
\usepackage{hyperref}

\def\R{{\mathbb R}}
\def\E{{\mathbb E}}

\def\K{{\mathbb K}}

\def\bQ{{\bf Q}}
\def\wt{\widetilde}

\def\<{\langle}
\def\>{\rangle}

\def\LL{{\cal L}}

\newtheorem{thm}{Theorem}[section]

\newtheorem{lem}[thm]{Lemma}

\newtheorem{defn}[thm]{Definition}
\newtheorem{rem}[thm]{Remark}

\numberwithin{equation}{section}

\allowdisplaybreaks[4]
\newcounter{counterConstant}
\newcommand{\const}[1]{
    \addtocounter{counterConstant}{1}
    \edef#1{\arabic{counterConstant}}
}

\begin{document}
\title{\bf Heat kernel estimates for  $\Delta+\Delta^{\alpha/2}$ under gradient perturbation}
\author{{\bf Zhen-Qing Chen\footnote{Research supported in part by NSF grant DMS-1206276 .}} \quad and \quad {\bf Eryan Hu}\footnote{Research supported in part by NFSC (11071138).}}
\date{}
\maketitle

\begin{abstract}
  For $\alpha \in (0,2)$ and $M > 0$, we consider a family of nonlocal operators $\{\Delta+a^\alpha\Delta^{\alpha/2}, a\in (0,M]\}$ on $\mathbb{R}^d$ under Kato class gradient perturbation. We establish the existence and uniqueness of their fundamental solutions, and derive their sharp two-sided estimates. The estimates give explicit dependence on $a$ and recover the sharp estimates for Brownian motion with drift as $a\to 0$. Each fundamental solution determines a conservative Feller process $X$. We characterize $X$ as the unique solution of the corresponding martingale problem as well as a L\'evy process with singular drift.
\end{abstract}



\section{Introduction}

Let $B$ be a Brownian motion on $\R^d$ with $\E [(B_t-B_0)^2]=2t$,
and $Y$ be a rotationally symmetric
$\alpha$-stable process on $\R^d$ that is independent of $B$.
Here $d \ge 1$ and $\alpha \in (0,2)$.
Then $B+Y$ is a symmetric L\'evy process that has both  diffusive and jumping components. Let $b$ be a bounded
$\R^d$-valued function on $\R^d$. Using Girsanov transform, it is easy to show that for every $a>0$,
there is a strong Markov process $X^{a, b}$ on $\R^d$ so that
\begin{equation}\label{e:1.1}
 dX^{a, b}_t = dZ^a_t +b(X^{a, b}_t)dt,
\end{equation}
where $Z^a$ is a L\'evy process that has the same distribution as $B+aY$.
The goal of this paper is to study the transition density function
$p^{a, b}(t, x, y)$ of
the strong Markov process $X^{a, b}$ and its two-sided sharp estimates.

Recall that  a rotationally symmetric $\alpha$-stable process on $\R^d$ is a L\'evy process $Y$ so that
\begin{equation*}
  \mathbb{E}_x[e^{i\xi (Y_t - Y_0)}] = e^{-t|\xi|^{\alpha}}  \text{ for every } x, \xi \in \mathbb{R}^d \text{ and } t > 0.
\end{equation*}
The infinitesimal generator of $Y$ is $\Delta^{\alpha/2} := -(-\Delta)^{\alpha/2}$, which is a prototype of nonlocal operator and can be written in the form
\begin{equation}\label{equ_alphaexp}
  \Delta^{\alpha/2}f(x) = \lim_{\varepsilon \rightarrow 0}\int_{|x-y|\ge \varepsilon} \mathcal{A}(d,-\alpha)\frac{f(y) - f(x)}{|x-y|^{d+\alpha}} dy, \quad f\in C_c^2(\mathbb{R}^d).
\end{equation}
Here $\mathcal{A}(d,-\alpha) := \alpha2^{\alpha-1}\pi^{-d/2}\Gamma((d+\alpha)/2)\Gamma(1-\alpha/2)^{-1}$ is a normalizing constant, with $\Gamma(\lambda):= \int_0^\infty t^{\lambda-1}e^{-t}dt$. Using It\^o's formula, one can see that the infinitesimal generator of $X^{a,b}$ is
$$
\mathcal{L}^{a,b} = \Delta+a^\alpha\Delta^{\alpha/2} + b\cdot \nabla .
$$

In this paper we will in fact study heat kernel estimates of $X^{a, b}$ not only for bounded drift function $b$ but also for $b$ in certain Kato class $\K_{d, 1}$ which can be unbounded; see Definition \ref{D:1.1}. When $b$ is in Kato class $\mathbb{K}_{d, 1}$, one can not obtain the strong Markov process $X^{a, b}$ from $B+aY$ through Girsanov transform. So we will do it in the other way around. We first construct and establish in Theorem \ref{T:main} the uniqueness of the fundamental solution $p^{a,b}(t, x, y)$ for operator $\LL^{a, b}$, and obtain its two-sided sharp estimates in Theorem \ref{T:1.3}. The heat kernel $p^{a, b}(t, x, y)$ determines a conservative Feller process $X^{a, b}$. We then show in Theorem \ref{T:1.5} that $X^{a, b}$ satisfies \eqref{e:1.1} through establishing the well-posedness of the martingale problem for $(\LL^{a,b}, C_c^\infty (\R^d))$ in Theorem \ref{T:1.4}. Moreover, we derive sharp two-sided estimates for $p^{a, b}(t, x, y)$ in such a way that gives the explicit dependence on $a$ so that when $a\to 0$, we can recover the sharp two-sided heat kernel estimates for Brownian motion with drift obtained in Zhang \cite{Zhang.1996, Zhang.1997}.

Brownian motions with drifts, which have $\Delta + b\cdot \nabla$ as their  infinitesimal generators, have been studied by many authors under various conditions; see \cite{CranstonZhao.1987, Zhang.1996, Zhang.1997}  and the references therein, where $b$ belongs to some suitable Kato class. In \cite{BogdanJakubowski.2007}, a fundamental solution to $\Delta^{\alpha/2}+b \cdot \nabla$ on $\R^d$ with $d\geq 2$ is constructed and its two-sided estimates derived. The uniqueness of the fundamental solution, the well-posedness of the martingale problem for $( \Delta^{\alpha/2}+b \cdot \nabla, C^\infty_c (\R^d))$ and its connection to stochastic differential equations are recently settled in \cite{ChenWang.2013a}. We also mention that relativistic stable processes with drifts have recently been studied in \cite{ChenWang.2014}.

We now describe the main results of this paper in more details. The L\'evy process  $Z^a$ has  infinitesimal generator $\LL^a:=\Delta+a^\alpha\Delta^{\alpha/2}$, and  L\'{e}vy intensity kernel
\begin{equation}\label{equ_Levydensity}
  J^a(x,y) = a^\alpha\mathcal{A}(d,-\alpha) |x-y|^{-(d+\alpha)},
\end{equation}
The kernel $J^a(x,y)$ determines a L\'{e}vy system for $Z^a$, which describes the jumps of the process $Z^a$. Let $p^a(t,x,y) := p^a(t,x-y)$ be the transition density function of $Z^a$ with respect to the Lebesgue measure on $\R^d$.
Clearly,  $p^a(t, z)$ is the smooth function determined by
\begin{equation}\label{equ_Xakernel}
  \int_{\mathbb{R}^d} p^a(t, z) e^{iz\cdot \xi} dz = e^{-t(|\xi|^2+a^\alpha |\xi|^\alpha)}, \quad \xi \in \mathbb{R}^d.
\end{equation}
The following sharp  two-sided estimates on $p^a (t, z)$, as stated in \cite[Theorem 1.1]{ChenKimSong.2011}, follows directly from \cite[Theorem 1.4]{ChenKumagai.2010} (see also \cite[Theorem 2.13]{SongVondravcek.2007}) by scaling.
\const{\CTpaBndI}\const{\CTpaBndII}\!\!\!\!\!\!\!\!\!\!\!\!\!\!\!\!\!\!\!\!\!\!\!\!\!\!\!
There exist constants $C_i\ge 1$, $i=\CTpaBndI,\CTpaBndII$, so that for all $a \in (0,\infty)$ and $(t,z) \in (0,\infty) \times \mathbb{R}^d $,
\begin{equation}\label{equ_paSharpBnds}
  \begin{split}
    C_{\CTpaBndI}^{-1}& (t^{-d/2}\wedge(a^\alpha t)^{-d/\alpha}) \wedge \left(t^{-d/2}e^{-C_{\CTpaBndII}|z|^2/t}+(a^\alpha t)^{-d/\alpha}\wedge \frac{a^\alpha t}{|z|^{d+\alpha}}\right)\\
    &\le p^a(t,z ) \le C_{\CTpaBndI} (t^{-d/2}\wedge(a^\alpha t)^{-d/\alpha}) \wedge \left(t^{-d/2}e^{-|z|^2/(C_{\CTpaBndII}t)}+(a^\alpha t)^{-d/\alpha}\wedge \frac{a^\alpha t}{|z|^{d+\alpha}}\right).
  \end{split}
\end{equation}

We can view $\mathcal{L}^{a,b}$ as the perturbation of $\LL^a$ by  $b\cdot \nabla$. So intuitively, the fundamental solution $p^{a,b}(t, x, y)$ of $\LL^{a, b}$ should be related to the fundamental solution $p^a(t, x, y)$ by the following formula
\begin{equation}\label{equ_Duhamel}
  p^{a,b}(t,x,y) = p^a(t,x,y) + \int_0^t\int_{\mathbb{R}^d} p^{a,b}(t-s,x,z) b(z)\nabla_z p^a(s,z,y) dzds
\end{equation}
for $t > 0$ and $x,y\in \mathbb{R}^d$. The above relation is a folklore and is called Duhamel's formula in literature.
Just as in \cite{BogdanJakubowski.2007,Zhang.1997}, applying  (\ref{equ_Duhamel}) recursively, it is reasonable to conjecture that $\sum_{k=0}^\infty p_k^{a,b}(t,x,y)$, if convergent, is a solution of \eqref{equ_Duhamel}, where $p_0^{a,b}(t,x,y) = p^a(t,x,y)$ and
\begin{equation}\label{equ_pndef}
  p_k^{a,b}(t,x,y) = \int_0^t\int_{\mathbb{R}^d} p_{k-1}^{a,b}(t-s,x,z) b(z)\nabla_z p^a(s,z,y) dzds  \text{ for } k \ge 1.
\end{equation}

We now give the definition of Kato class  $\mathbb{K}_{d, 1}$. For a function $f=(f_1, \dots, f_k):\mathbb{R}^d \rightarrow \mathbb{R}^k$ and $d\geq 2$, define
\begin{align*}
  M_f(r) = \sup_{x \in \mathbb{R}^d} \int_{|x - y| < r} \frac{|f(y)|}{|x - y|^{d - 1}}dy
  \quad \text{ for } r > 0.
\end{align*}

\begin{defn}\label{D:1.1}
  A function $f=(f_1, \dots, f_k):\mathbb{R}^d \rightarrow \mathbb{R}^k$ is said to be in  Kato class $\mathbb{K}_{d, 1}$ if  $ \lim_{r \downarrow 0} M_f (r) = 0$ when $d\geq 2$,  and bounded if $d=1$.
\end{defn}

It is easy to see that any bounded function is in Kato class $\mathbb{K}_{d,1}$ and, for $d\geq 2$, $L^p(\R^d)\subset \mathbb{K}_{d,1}$ for any $p>d$ by H\"older inequality. On the other hand, any function in $\mathbb{K}_{d,1}$ is locally integrable on $\R^d$.

For an integer $k \geq 1$, let $C_c^{k}(\mathbb{R}^d)$ denote the space of all continuous functions on $\mathbb{R}^d$ with compact supports that have continuous derivatives up to and including $k$th-order, and set $C_c^\infty(\mathbb{R}^d) = \cap_{k=1}^\infty C_c^k(\mathbb{R}^d)$. Denote by $C_\infty(\mathbb{R}^d)$ the space of continuous functions on $\mathbb{R}^d$ vanishing at the infinity, equipped with supremum norm. The following are the first two main results of this paper.

\begin{thm}\label{T:main}
  Suppose that $M > 0$ and $b=(b_1, \dots, b_d) \in \mathbb{K}_{d,1}$. For every $a\in (0,M]$, there is a unique positive jointly continuous function $p^{a,b}(t,x,y)$ on $(0,\infty)\times \mathbb{R}^d \times \mathbb{R}^d$ that satisfies (\ref{equ_Duhamel}) with $p^{a,b}(t,x,y) \le c_1 p^a(t,x,y)$ both on $(0,t_0]\times \mathbb{R}^d \times \mathbb{R}^d$ for some constants $c_1,t_0>0$, and that
  \begin{equation}\label{e:1.8}
    p^{a,b}(t+s,x,y) = \int_{\mathbb{R}^d} p^{a,b}(t,x,z) p^{a,b}(s,z,y) dz  \quad \text{ for } t, s > 0, x,y \in \mathbb{R}^d.
  \end{equation}
  Moreover, the following hold.
  \begin{enumerate}
    \item[\rm (i)] There is a constant $t_*=t_*(d,\alpha,M,b) > 0$, depending on $b$ only via the rate at which $M_b(r)$ goes to zero, such that
        \begin{equation}\label{e:1.9}
          p^{a,b}(t,x,y) = \sum_{k=0}^\infty p_k^{a,b}(t,x,y) \quad \text{ on } (0,t_*]\times \mathbb{R}^d \times \mathbb{R}^d,
        \end{equation}
        where $p_k^{a,b}(t,x,y)$ is defined by (\ref{equ_pndef}).

    \item[\rm (ii)] $p^{a,b}(t,x,y)$ satisfies (\ref{equ_Duhamel}) on $(0,\infty)\times \mathbb{R}^d \times \mathbb{R}^d$.

    \item[\rm (iii)] (Conservativeness)
       $\int_{\mathbb{R}^d} p^{a,b}(t,x,y) dy = 1$  for  every $t > 0$ and $ x \in \mathbb{R}^d$.

    \item[\rm (iv)] for every $f\in C_c^{\infty}(\mathbb{R}^d)$ and $g \in C_\infty(\mathbb{R}^d)$,
        \begin{equation}\label{equ_generator}
          \lim_{t\rightarrow 0}\int_{\mathbb{R}^d} \frac{P^{a,b}_tf(x) -  f(x)}{t} g(x) dx = \int_{\mathbb{R}^d} \mathcal{L}^{a,b}f(x) g(x) dx,
        \end{equation}
        where $P^{a,b}_tf(x) = \int_{\mathbb{R}^d}p^{a,b}(t,x,y) f(y) dy$.
  \end{enumerate}
\end{thm}

\medskip
Here and after, the meaning of the phrase ``depending on $b$ only via the rate at which $M_b(r)$ goes to zero" is that the statement is true for any  $\mathbb{R}^d$-valued function $\tilde{b}$ on $\mathbb{R}^d$ with $M_{\tilde{b}}(r) \le M_b (r)$ for all $r > 0$. In this paper, we use $:=$ as a way of definition. For $a, b\in \R$,  $a\wedge b:= \min\{a, b\}$ and $a \vee b:= \max\{a , b\}$.  For constants $a, \beta> 0$, we define
\const{\CTpabLc}
\const{\CTpabLexp}
\const{\CTpabUc}
\const{\CTpabUexp}
\begin{equation}\label{e:1.11}
  q_{d,\beta}^a(t, z) = t^{-d/2}\exp\left(-\frac{\beta|z|^2}{t}\right) + t^{-d/2}\wedge \frac{a^\alpha t}{|z|^{d+\alpha}}
  \quad \text{ for } t> 0, \  z\in \mathbb{R}^d.
\end{equation}

\begin{thm}\label{T:1.3}
  For every $M > 0$ and $T> 0$, there are constants $C_i = C_i(d,\alpha,M),i=\CTpabLexp,\CTpabUexp$ and $C_j = C_j(d,\alpha,M,T,b),j=\CTpabLc,\CTpabUc$ depending on $b$ only via the rate at which $M_b(r)$ goes to zero, such that for all $a\in (0,M]$ and $(t,x,y) \in (0,T]\times \mathbb{R}^d\times \mathbb{R}^d$,
  \begin{equation}\label{e:1.12}
    C_{\CTpabLc} q_{d,C_{\CTpabLexp}}^a(t,x-y) \le p^{a,b}(t,x,y) \le C_{\CTpabUc} q_{d,C_{\CTpabUexp}}^a(t,x-y).
  \end{equation}
\end{thm}

 The heat kernel upper bound estimate of $p^{a,b}\! (t,x,y)$ is obtained by estimating each $p^{a,b}_k\! (t, x, y)$ in \eqref{e:1.9}.  It relies on a key estimate obtained in Lemma \ref{thm_ace}, which can be regarded as an analogy of the so called 3P estimate in \cite[Lemma 3.1]{Zhang.1997} and \cite[Lemma 13]{BogdanJakubowski.2007}. However, unlike the case in \cite{Zhang.1997} where there is only Gaussian term coming from Brownian motion and the case in \cite{BogdanJakubowski.2007} where there is only polynomial term coming from symmetric stable process, there are many new difficulties to overcome as we have to deal with a mixture of them. It seems to be difficult to establish the positivity of $p^{a,b}(t, x, y)$ directly from the estimates of $p^{a,b}_k(t,x, y)$ as was done in \cite{BogdanJakubowski.2007} for the symmetric stable process case. Following \cite{ChenWang.2012}, we derive the positivity of $p^{a,b}(t, x, y)$ by using the Hille-Yosida-Ray theorem when $b$ is bounded and continuous. For general $b$ in Kato class $\mathbb{K}_{d, 1}$, we approximate $b$ by a sequence of smooth $b_n$. For the lower bound of $p^{a,b}(t,x,y)$ in Theorem \ref{T:1.3}, we identify and use the L\'{e}vy system of the Feller process $\{X^{a,b}_t,t\ge 0, \mathbb{P}^{a,b}_x, x\in\mathbb{R}^d\}$ associated with $\{P^{a,b}_t, t\ge 0\}$ to get the polynomial part  (see Lemma \ref{thm_pablowpoly}), and use a chaining argument   to get the Gaussian part (see Lemma \ref{thm_pablowexp}).

 Let $\mathbb{D}([0,\infty),\mathbb{R}^d)$ be the space of right continuous $\R^d$-valued functions on $[0, \infty)$ having left limits equipped with Skorokhod topology, and let $X_t$ be the coordinate map on $\mathbb{D}([0,\infty),\mathbb{R}^d)$. A probability measure $\bQ$ on $\mathbb{D}([0,\infty),\mathbb{R}^d)$ is said to be a solution to the martingale problem for $(\mathcal{L}^{a,b},C^\infty_c(\mathbb{R}^d))$ with initial value $x\in\mathbb{R}^d$ if $\bQ (X_0 = x) = 1$ and for every $f \in C^\infty_c(\mathbb{R}^d)$ and $t > 0$, $\int_0^t | \mathcal{L}^{a,b} f(X_s)| ds < \infty$ $\bQ$-a.s. and
 \begin{equation*}
   M^f_t := f(X_t) - f(X_0) - \int_0^t \mathcal{L}^{a,b} f(X_s) ds
 \end{equation*}
 is a $\bQ$-martingale. The martingale problem for $(\mathcal{L}^{a,b},C^\infty_c(\mathbb{R}^d))$ with initial value $x\in \mathbb{R}^d$ is said to be well-posed if it has a unique solution.

\begin{thm}\label{T:1.4}
  The martingale problem for $(\mathcal{L}^{a,b},C^\infty_c(\mathbb{R}^d))$ is well-posed for every initial value $x\in \R^d$. These martingale problem solutions $\{\mathbb{P}_x, x \in \mathbb{R}^d\}$ form a strong Markov process $X$, which has $p^{a,b}(t,x,y)$ of Theorem \ref{T:main} as its  transition density function  with respect to the Lebesgue measure on $\mathbb{R}^d$.
\end{thm}

We now connect the strong Markov process in Theorem \ref{T:1.4} to solution of SDE \eqref{e:1.1}.

\begin{thm}\label{T:1.5}
  For each $x \in \mathbb{R}^d$, SDE \eqref{e:1.1} has a unique weak solution
  with initial value $x$. Moreover, weak solutions with different starting points can be constructed on $\mathbb{D}([0,\infty),\mathbb{R}^d)$, and the process $Z^a$ in \eqref{e:1.1} can be chosen in such a way that it is the same for all starting point $x \in \mathbb{R}^d$. The law of the weak solution to \eqref{e:1.1} is the unique solution to the martingale problem for $(\mathcal{L}^{a,b},C^\infty_c(\mathbb{R}^d))$.
\end{thm}

\begin{rem} \rm Brownian motion with measure-valued singular drift on $\R^d$,
where $d\geq 2$ and the $\R^d$-valued drift $b$ is replaced by a measure $\mu=(\mu_1, \dots, \mu_d)$ in Kato class $\K_{d, 1}$,
 is introduced and constructed in Bass and Chen \cite{BC}.
With the two-sided heat kernel estimates from Theorem \ref{T:1.3}, one can easily construct L\'evy process $Z^a$ with singular  measure-valued drift
$\mu=(\mu_1, \dots, \mu_d)$ in the sense of \cite{BC} and obtain its sharp two-sided heat kernel estimates. The key is to note that the two-sided heat kernel estimates in Theorem \ref{T:1.3} depend  on the drift $b$ only through its upper bound of $M_b ( r)$ so we can approximate the measure-valued drift $\mu$ by a sequence of function-valued drifts whose Kato norms are uniformly controlled by that of $\mu$.
 Here are the details. Suppose $\mu=(\mu_1, \dots, \mu_d)\in K_{d, 1}$, that is,
$$
 \lim_{r\to 0}  M_\mu  (r) := \lim_{r\to 0} \sup_{x\in \R^d} \int_{|y-x|<r} \frac{1}{|y-x|^{d-1}} |\mu|(dy)=0 .
$$
Let $\varphi \in C^\infty_c (\R^d)$ with $\varphi \geq 0$ and $\int_{\R^d} \varphi (x) dx=1$.
We can approximate  $\mu$ by $b_n (x)dx = \varphi_n *\mu (dx)$, where
$\varphi_n (x)= n^{d} \varphi (nx)$. Note that  $\{b_n; n\geq 1\} \subset \K_{d, 1}$ with
$M_{b_n} (r ) \leq M_\mu ( r )$ for every $n\geq 1$ and $r>0$.
Denote by $p^{b_n}(t, x, y)$ and $X^n$ the heat kernel for $\LL^{b_n}$ and its corresponding Feller process.
The two-sided heat kernel estimate \eqref{e:1.12} holds uniformly in $n$ for
$p^{a, b_n}(t, x, y)$
on $(0, 1] \times \R^d \times \R^d$.  Similar to that of \cite[Theorem 3.9]{KS},
one  can show that $\{p^{a, b_n}(t, x, y); t>0, x, y\in \R^d\}$
converges locally uniformly to $p^{a, \mu}(t, x, y)$. It is easy to verify that
$p^{a, \mu}(t, x, y)$ is a positive kernel which enjoys the two-sided estimates
\eqref{e:1.12}. Moreover, it
satisfies the Chapman-Kolmogorov equation
and $\int_{\R^d} p^{a, \mu}(t, x, y)dy=1$ for all $t>0$ and $x\in \R^d$. The kernel $p^{a, \mu}(t, x, y)$ determines a Feller process $X$.
It is not hard to verify that it is a L\'evy process $Z^a$ with measure-valued drift $\mu$ in the sense of Bass and Chen \cite{BC}. See \cite{KS} for the case
when $Z^a$ is a rotationally symmetric stable process on $\R^d$.  \qed
\end{rem}

The rest of this paper is organized as follows. In Section 2, we recall some properties  of $p^a(t,x,y)$ and derive its gradient estimates, as well as properties of functions in Kato class $\mathbb{K}_{d,1}$. In Section 3, we construct $p^{a,b}(t,x,y)$ using the series of $p_k^{a,b}(t,x,y)$ and prove Theorem \ref{T:main} through a series of lemmas except the positivity of $p^{a,b}(t,x,y)$. In addition, we derive the upper bound of $|p^{a,b}(t,x,y)|$. The positivity of $p^{a,b}(t,x,y)$ is shown in Section 4, where we use the fact that $\{P_t^{a,b}, t\ge 0\}$ is Feller semigroup, that is, a strongly continuous semigroup in $C_\infty(\mathbb{R}^d)$. In Section 5, we determine the L\'evy system of the Feller process $X^{a,b}$ associated with the Feller semigroup $\{P_t^{a,b}, t\ge 0\}$. We then use it to derive the lower bound estimate of $p^{a,b}(t, x, y)$. In Section 6, we prove Theorem \ref{T:1.4} and Theorem \ref{T:1.5}.

For convenience, in the rest of this paper, we assume $d\geq 2$. When $d=1$, it can be treated in a similar but simpler way as the drift $b$ would be bounded. Throughout this paper, unless stated otherwise, we use $C_1,C_2,\cdots ,$ to denote positive constants whose value are fixed throughout the paper, while using $c_1,c_2,\cdots, $ to denote positive constants whose exact value are unimportant and whose value can change from one appearance to another. We use notation $c = c(d,\alpha,\cdots )$ to indicate that this constant depends only on $d,\alpha,\cdots$.  For two non-negative functions $f,g$, the notation $f \stackrel{c}{\lesssim} g$ means that $f \le cg$ on their common domains of definition while $f \stackrel{c}{\asymp} g$ means that $c^{-1} g \le f \le c g$. We also write mere $\lesssim$ and $\asymp$ if $c$ is unimportant or understood. For reader's convenience, we summarize the notation of functions that will appear many times throughout this paper. For $t>0$ and $x, y\in \R^d$,
\begin{eqnarray}
  p^a (t, x, y)&=& p^a (t, x-y) :  \hbox{ the transition density function of } B+aY \nonumber \\
  g_{d, \beta} (t,x, y ) &=& g_{d, \beta} (t, x-y):= t^{-d/2}\exp\left(-\frac{\beta|x-y |^2}{t}\right),    \label{equ_gDef} \\
  g_d (t, x, y)&=& g_d (t, x-y):= (4\pi)^{-d/2} g_{d, 1/4} (t, x-y),  \nonumber \\
  q^a_{d, \beta} (t, x, y) &=& q^a_{d, \beta} (t, x- y) :=  g_{d, \beta} (t, x-y) + t^{-d/2}\wedge \frac{a^\alpha t}{|x-y|^{d+\alpha}}.\label{e:1.13}
\end{eqnarray}

\section{Preliminaries}
The following is a direct consequence of \eqref{equ_paSharpBnds}; see  \cite[Corollary 1.2]{ChenKimSong.2011}.

\const{\CTpaLc}
\const{\CTpaLexp}
\const{\CTpaUc}
\const{\CTpaUexp}

\begin{thm}\label{thm_pabound}
  For any $M > 0$ and $T > 0$, there exist constants $C_i,i = \CTpaLexp,\CTpaUexp$ and $C_j = C_j(d,\alpha,M,T)$, $j=\CTpaLc,\CTpaUc$ such that for all $a\in (0,M]$ and $(t,x ) \in (0,T]\times \mathbb{R}^d $,
  \begin{equation*}
    C_{\CTpaLc}q_{d,C_{\CTpaLexp}}^a(t,x ) \le p^a(t,x ) \le C_{\CTpaUc}q_{d,C_{\CTpaUexp}}^a(t,x ).
  \end{equation*}
\end{thm}

It is easy to see that for any $\theta > 0$, there is a positive constant $c_1 = c_1(d,\beta,\theta)$ such that
\begin{equation}\label{equ_Gaussianup}
  g_{d, \beta}(t,x ) \le  t^{-d/2} \wedge \frac{c_1t^\theta}{|x |^{d+2\theta}},\quad t> 0 \text{ and } x \in \mathbb{R}^d,
\end{equation}
which will be frequently used in the rest of this paper.

\const{\CTqEqv}
Recall the definition of $q_{d,\beta}^a(t,x )$ in \eqref{e:1.11}. There is a constant $C_{\CTqEqv} = C_{\CTqEqv}(\alpha,M,T,\beta)$ such that for all $a\in (0,M]$ and all $(t, z)\in (0,T]\times\mathbb{R}^d$,
\begin{equation}\label{equ_qeqvlnt}
  q_{d,\beta}^a(t,z) \stackrel{C_{\CTqEqv}}{\asymp} g_{d, \beta}(t,z)+\frac{a^{\alpha}t}{|z|^{d+\alpha}}{\bf 1}_{\{|z|^2 \ge t\}} .
\end{equation}
Indeed, $ t^{-d/2}\wedge \frac{a^{\alpha}t}{|z|^{d+\alpha}}  \le t^{-d/2}\le e^{\beta}g_{d, \beta}(t,z)$ when $|z|^2 < t$. Thus
\begin{align}\label{equ_qright}
  q_{d,\beta}^a(t,z) &\stackrel{e^{\beta} +1 }{\lesssim} g_{d, \beta}(t,z)+\frac{a^{\alpha}t}{|z|^{d+\alpha}}{\bf 1}_{\{|z|^2 \ge t\}}
  \qquad \hbox{for   } a  ,t>0 \hbox{ and } z\in\mathbb{R}^d .
\end{align}
On the other hand, for  $a\in (0,M]$ and $t\in (0,T]$,
\begin{equation*}
  \frac{a^{\alpha}t}{|z|^{d+\alpha}} \le M^\alpha t^{-d/2+1-\alpha/2} \le M^\alpha T^{1-\alpha/2} t^{-d/2}
  \quad \hbox{if }  |z|^2 \ge t ,
\end{equation*}
and so
\begin{equation}\label{equ_qleft}
  g_{d, \beta}(t,z ) + \frac{a^{\alpha}t}{|z|^{d+\alpha}}{\bf 1}_{\{|z|^2\ge t\}} \stackrel{M^\alpha T^{1-\alpha/2}\vee 1}{\lesssim} q_{d,\beta}^a(t,z).
\end{equation}
The claim (\ref{equ_qeqvlnt}) now follows from (\ref{equ_qright}) and (\ref{equ_qleft}) with $C_{\CTqEqv} = (e^\beta+1) \vee (M^\alpha T^{1-\alpha/2}\vee 1)$.

\medskip

When there is no danger of confusion, for $x\in \R^d$ and integer $k\geq 1$, for simplicity, we write $q^a_{d+k, \beta} (t, x)$ for $q^a_{d+k, \beta} (t, \wt x)$, where $\wt x:=(x, 0, \dots, 0)\in \R^{d+k}$. Same convention will apply to
function $g_{d, \beta}(t, x)$.

\medskip
The following theorem gives the two-sided estimate of $|\nabla_x p^a(t,x )|$. In this paper, only its upper bound will be used.

\const{\CTnablapaU}

\begin{thm}\label{thm_Gradpabound}
  For any $M > 0$ and $T > 0$, there is a positive constant $C_{\CTnablapaU}=C_{\CTnablapaU}(d,\alpha,M,T)$ such that for all $a \in (0,M]$ and $(t,x ) \in (0,T]\times\mathbb{R}^d$,
  \begin{align*}
    2\pi C_{\CTpaLc}q_{d+2,C_{\CTpaLexp}}^{a}(t,x )|x | \le |\nabla_x p^a(t,x )| \le C_{\CTnablapaU} q_{d+1,3C_{\CTpaUexp}/4}^a(t,x ).
  \end{align*}
\end{thm}

\begin{proof}
 It is well-known that, for each $t>0$, $x\mapsto p^a (t, x)$ attains its maximum at $x=0$ so we have $\nabla p^a(t, 0) = 0$. So it suffices to consider $x\in \R^d\setminus \{0\}$. Recall that $g_d (t, z) = (4\pi t)^{-d/2} e^{-|z|^2/(4t)}$, which is the transition density function of Brownian motion $B$. Let $S_t$ be the $\alpha/2$-stable subordinator at time $t$, independent of $B$, and $\eta_t^a(u)$ be the density function of $a^2S_t$. The L\'evy process $Z^a$ can be realized as a subordination of Brownian motion $B$; that is, $\{Z^a_t; t\geq 0\}$ has the same distribution as $\{B_{t+a^2S_t}; t\geq 0\}$. Thus
 \begin{align*}
   p^a(t,x) &= \int_t^{+\infty} g_d (u, x) \mathbb{P}(t + a^2S_t \in du)  = \int_t^{+\infty} g_d (u, x) \eta_t^a(u - t) du,
 \end{align*}
 and so
 \begin{align*}
   \nabla_x p^a(t,x) = \nabla_x \int_t^{\infty} g_d (u, x)\eta_t^a(u - t) du.
 \end{align*}
 Let $e_j = (0,\cdots,0,1,0,\cdots,0)$, where $1$ is on $j^{th}$ place. Let $x\in \R^d \setminus \{0\}$ and set $s \in (-|x|/2,|x|/2)$. By the mean-value theorem, there exists $\xi \in (-|s|,|s|)$ such that
 \begin{align*}
   \left|\frac{g_d (u, x + se_j) - g_d (u, x)}{s}\right| &= \left|\frac{\partial}{\partial x_j}g_d (u, x + \xi e_j)\right| = \left|\frac{x_j + \xi}{2u}g_d (u, x + \xi e_j)\right|   \\
   &\le \frac{|x|g_d (u, x/2)}{u} \le c(d)|x|^{-d-1},
 \end{align*}
 where $c(d)$ is a positive constant depending only on $d$. Since $\int_t^{\infty} c(d)|x|^{-d-1}\eta_t(u - t) du < \infty$, we have by the dominated convergence theorem
  \begin{equation}\label{equ_ghk}
    \begin{split}
      \nabla_x p^a(t,x) &= \int_t^{\infty} \nabla_x  g_d (u, x)\eta_t^a(u - t) du = \int_t^{\infty} -\frac{xg_d (u, x)}{2u}\eta_t^a(u - t) du = -2\pi xp_{(d+2)}^a(t,\tilde x),
    \end{split}
  \end{equation}
  where $\tilde x:=(x, 0, 0) \in \mathbb{R}^{d+2}$  and $p^a_{d+2}(t,\tilde x)$ is the transition density function of  $Z^a$ in dimension $d+2$. Thus, by Theorem \ref{thm_pabound}, we have
  \begin{equation}\label{e:2.6}
    2\pi C_{\CTpaLc}q_{d+2,C_{\CTpaLexp}}^a(t,  x)|x| \le |\nabla_x p^a(t,x)| \le 2\pi C_{\CTpaUc} q_{d+2,C_{\CTpaUexp}}^a(t,  x)|x|.
  \end{equation}
  Note that for all $t > 0$ and $x \in \mathbb{R}^d$,
  \begin{align*}
    t^{-(d+2)/2} \exp\left(-\frac{C_{\CTpaUexp}|x |^2}{t}\right)|x | &=t^{-(d+1)/2}\exp\left(-\frac{3C_{\CTpaUexp}}{4}\frac{|x |^2}{t}\right) \cdot \frac{|x|}{t^{1/2}}\exp\left(-\frac{C_{\CTpaUexp}}{4}\frac{|x |^2}{t}\right)\\
    & \le  \sqrt{\frac{2}{C_{\CTpaUexp}e}}t^{-(d+1)/2} \exp\left(-\frac{3C_{\CTpaUexp}}{4}\frac{|x |^2}{t}\right).
  \end{align*}
   This together with \eqref{equ_qeqvlnt} and \eqref{e:2.6} proves the theorem with
    $C_{\CTnablapaU} := 2\pi C_{\CTpaUc} C_{\CTqEqv} \left(\sqrt{2/(C_{\CTpaUexp}e)}\vee 1\right)$.
\end{proof}

For $\beta > \frac{1}{2}$ and a function $f$ on $\mathbb{R}^d$, define for $r>0$ and $x\in \R^d$,
\begin{equation*}
  H^\beta(r,x )= \frac{1}{|x  |^{d - 1}}\wedge \frac{r^\beta}{|x|^{d-1+ 2\beta}} \quad \hbox{and} \quad
  H_f^{\beta}(r,x) = \int_{\mathbb{R}^d} |f(y)| H^\beta (r,x-y) dy .
\end{equation*}

\const{\CTHM}
\begin{lem}\label{thm_HMb}
  Assume $\beta > \frac{1}{2}$. There is a constant $C_{\CTHM} = C_{\CTHM}(d,\beta)$ so that
  \begin{equation}\label{equ_HMb}
   M_f(\sqrt{r})\leq  H_f^\beta(r,x) \le C_{\CTHM}M_f(\sqrt{r}),
  \end{equation}
  for every $r > 0,x \in \mathbb{R}^d$ and for every $f$ on $\mathbb{R}^d$.
  Consequently, $f \in \mathbb{K}_{d,1}$ if and only if
  \begin{equation*}
    \lim_{r\downarrow 0} \sup_{x\in \mathbb{R}^d} H_f^{\beta} (r,x) = 0.
  \end{equation*}
\end{lem}

The lower bound in \eqref{equ_HMb} is trivial. The proof of the upper bound in \eqref{equ_HMb} is  almost the same as that for \cite[Lemma 11 and  Corollary 12]{BogdanJakubowski.2007} except with 2 in place of $\alpha$ there. So we omit its details.

Let
\begin{align*}
  N^{\beta}(r,x ) =\int_0^r g_{d+1, \beta} (s, x)ds = \int_0^r s^{-(d+1)/2}\exp\left(-\frac{\beta|x |^2}{s}\right)ds,\quad r> 0, x  \in \mathbb{R}^d.
\end{align*}

\begin{lem}\label{thm_Kequi}
  $f \in \mathbb{K}_{d,1}$ if and only if
  \begin{equation}\label{e:2.8}
    \lim_{r\downarrow 0} \sup_{x\in\mathbb{R}^d} \int_{\mathbb{R}^d} |f(y)| N^\beta(r,x-y) dy = 0  \text{ for all } \beta > 0.
  \end{equation}
\end{lem}

\begin{proof} Condition \eqref{e:2.8} is introduced in \cite{Zhang.1996}.
Its equivalence to the $\mathbb{K}_{d,1}$ condition is proved in \cite[Proposition 2.3]{KimSong.2006}. For reader's convenience, we give a short proof here.

 By a change of variable $t=\beta|x|^2/s$, we have
  \begin{equation}\label{equ_Nequal}
    N^{\beta}(r,x )= \frac1{\beta^{(d-1)/2}|x|^{d-1}}
    \int_{\beta|x|^2/r}^\infty  t^{(d-3)/2}e^{-t}  dt.
  \end{equation}
Thus
  \begin{equation}\label{equ_Nupper}
  c_1 (d, \beta) \frac{1}{|x|^{d-1}} {\bf 1}_{\{|x|\leq \sqrt{r}\}}
  \leq N^{\beta}(r,x )\leq c_2 (d, \beta  ) H^1 (r, x)
  \end{equation}
The equivalence now follows from Lemma \ref{thm_HMb}.
\end{proof}

\section{Construction and upper bound estimates}

By \cite[Lemma 3.1]{Zhang.1997} and its proof, we have the following lemma. Recall that $g_{d, \beta}(t,x,y):=g_{d, \beta}(t, x-y)$ is defined by (\ref{equ_gDef}), and define $H^\beta (r, x, y)= H^\beta (r, x-y)$ and $N^\beta(r,x,y) = N^\beta(r,x-y)$.

\begin{lem}\label{thm_Gaussian3P}
  For any $0<\beta_1<\beta_2<\infty$, there exist constants $C_g = C_g(d,\beta_1/\beta_2)$ and $C_\beta = min\{\beta_2-\beta_1,\beta_1/2\}$ such that for all $t > 0$ and $x,y,z\in \mathbb{R}^d$,
  \begin{align*}
    \int_0^t g_{d, \beta_1}(t-s,x,z)s^{-1/2}g_{d, \beta_2}(s,z,y) ds \le C_g(N^{C_\beta}(t,x,z)+N^{C_\beta}(t,z,y)) g_{d, \beta_1}(t,x,y).
  \end{align*}
\end{lem}

\const{\CTtherepI}
\const{\CTtherepII}
In the rest of this paper, we assume $b \in \mathbb{K}_{d,1}$ and let $\gamma = (1+\alpha \wedge 1)/2$. The following lemma plays an important role in this paper and it is an analogy of \cite[Lemma 13]{BogdanJakubowski.2007} or \cite[Lemma 3.1]{Zhang.1997}.
\begin{lem}\label{thm_ace}
  Suppose $M > 0$ and $T > 0$. For any $0< \beta_1< \beta_2 <\infty$, there is a positive constant $C_{\CTtherepI}=C_{\CTtherepI}(d,\alpha,M,T,\beta_1,\beta_2)$ such that for all $a\in (0,M]$ and $(t,x,y,z) \in (0,T]\times \mathbb{R}^d \times \mathbb{R}^d\times \mathbb{R}^d$,
  \begin{equation}\label{equ_3q}
    \int_0^t q_{d,\beta_1}^a(t - s,x,z) q_{d+1,\beta_2}^a(s,z,y) ds \le C_{\CTtherepI} (H^{\gamma}(t,x,z) + H^{\gamma}(t,z,y)) q_{d,\beta_1}^a(t,x,y).
  \end{equation}
  Consequently, there is a positive constant $C_{\CTtherepII}=C_{\CTtherepII}(d,\alpha,M,T)$ such that for all $a\in (0,M]$ and $(t,x,y)\in (0,T]\times \mathbb{R}^d\times \mathbb{R}^d$,
  \begin{equation}\label{equ_3p}
    \int_0^t\int_{\mathbb{R}^d} q_{d,\beta_1}^a(t - s,x,z)|b(z)|q_{d+1,\beta_2}^a(s,z,y)dz ds\le C_{\CTtherepII}M_b(\sqrt{t})q_{d,\beta_1}^a(t,x,y),
  \end{equation}
\end{lem}
\begin{proof}
  We first verify (\ref{equ_3q}). By (\ref{equ_qeqvlnt}), for all $(t,x,y) \in (0,T]\times \mathbb{R}^d \times \mathbb{R}^d$, there is a constant $c_1 = c_1(\alpha,M,T,\beta_1,\beta_2)$ such that
  \begin{eqnarray*}
    I&:=&\int_0^t q_{d,\beta_1}^a(t - s,x,z) q_{d+1,\beta_2}^a(s,z,y) ds\\
    &\stackrel{c_1}{\lesssim}& \int_0^t \!\!\left(g_{d,\beta_1}(t-s,x,z) \!+\! \frac{a^{\alpha}(t-s)}{|x-z|^{d+\alpha}} {\bf 1}_{\{|x-z|^2\ge t-s\}}\right)\!\!\left(\frac{g_{d,\beta_2}(s,z,y)}{s^{1/2}} \!+\! \frac{a^{\alpha}s}{|z-y|^{d+1+\alpha}} {\bf 1}_{\{|z-y|^2\ge s\}}\right) ds\\
    &=& \int_0^t g_{d,\beta_1}(t-s,x,z) \frac{g_{d,\beta_2}(s,z,y)}{s^{1/2}} ds
     + \int_0^t g_{d,\beta_1}(t-s,x,z) \frac{a^{\alpha}s}{|z-y|^{d+1+\alpha}} {\bf 1}_{\{|z-y|^2\ge s\}} ds \\ &&
     + \int_0^t \frac{a^{\alpha}(t-s)}{|x-z|^{d+\alpha}} {\bf 1}_{\{|x-z|^2\ge t-s\}} \frac{g_{d,\beta_2}(s,z,y)}{s^{1/2}} ds\\
    &&+ \int_0^t \frac{a^{\alpha}(t-s)}{|x-z|^{d+\alpha}}{\bf 1}_{\{|x-z|^2\ge t-s\}} \frac{a^{\alpha}s}{|z-y|^{d+1+\alpha}}{\bf 1}_{\{|z-y|^2\ge s\}} ds\\
    &=:& I_1 + I_2+ I_3 + I_4.
  \end{eqnarray*}
  We will treat each term separately. First, by Lemma \ref{thm_Gaussian3P}, there are constants $c_2 = c_2(d,\beta_1/\beta_2)$ and $c_3 = c_3(\beta_2-\beta_1),\beta_1/2)$ such that $I_1 \le c_2 \left( N^{c_3}(t,x,z)+N^{c_3}(t,z,y)\right) g_{d, \beta_1}(t,x,y)$,
  while
  \begin{eqnarray} \label{equ_I2t}
    I_2 &=& \left( \int_0^{t/2} + \int_{t/2}^t \right) g_{d,\beta_1}(t-s,x,z) \frac{a^{\alpha}s}{|z-y|^{d+1+\alpha}} {\bf 1}_{\{|z-y|^2\ge s\}} ds \notag\\
    &\stackrel{2^{d+\alpha} M^\alpha}{\lesssim}& t^{-d/2}\int_0^{(t/2)\wedge |z-y|^2} \frac{s}{|z-y|^{d+1+\alpha}}ds \notag\\
    && + t^{-d/2}{\bf 1}_{\{|z-y|^2\ge t/2\}} \int_0^{t/2} s^{-d/2}\exp\left(-\frac{\beta_1|x-z|^2}{s}\right)\frac{t^{1-\alpha/2}}{|z-y|}ds \notag\\
    &\stackrel{2}{\lesssim}& t^{-d/2} \left(\frac{1}{|z-y|^{d-3+\alpha}}\wedge \frac{t^{2}}{|z-y|^{d+1+\alpha}}\right) \notag\\
    && + t^{-d/2+1-\alpha/2}\int_0^{t/2} s^{-(d+1)/2} \exp\left(-\frac{\beta_1|x-z|^2}{s}\right)ds  \notag\\ &\stackrel{T^{1-\alpha/2}}{\lesssim}& t^{-d/2}\left(\frac{1}{|z-y|^{d-1}}\wedge \frac{t^{(1+\alpha)/2}}{|z-y|^{d+\alpha}}\right)+ t^{-d/2} N^{\beta_1}(t,x,z)\notag\\
    &=& t^{-d/2}\left(N^{\beta_1}(t,x,z) + H^{(1+\alpha)/2}(t,z,y) \right).
  \end{eqnarray}
  On the other hand, if $|x-z| \ge |z-y|$, then  $2|x-z| \ge |x-z|+|z-y|\ge |x-y|$, and so
  \begin{eqnarray}\label{equ_I2zy}
    I_2 &\stackrel{c_4(d,\alpha,\beta_1)}{\lesssim}& \int_0^{t\wedge |z-y|^2} \frac{(t-s)^{\alpha/2}}{|x-z|^{d+\alpha}}\frac{a^\alpha s}{|z-y|^{d+1+\alpha}} ds \notag\\
    &\stackrel{c_5(d,\alpha)}{\lesssim}& \frac{a^\alpha t}{|x-y|^{d+\alpha}} \int_0^{t\wedge |z-y|^2} t^{\alpha/2 -1}\frac{s}{|z-y|^{d+1+\alpha}} ds \notag\\
    &\stackrel{2}{\lesssim}& \frac{a^\alpha t}{|x-y|^{d+\alpha}} t^{\alpha/2-1} \frac{t^2\wedge |z-y|^4}{|z-y|^{d+1+\alpha}} \notag\\
    &\le& \frac{a^\alpha t}{|x-y|^{d+\alpha}} H^{1+\alpha/2}(t,z,y).
  \end{eqnarray}
  If $|x-z|< |z-y|$, then  $2|z-y| \ge |x-y|$  and
  \begin{eqnarray}\label{equ_I2xz}
    I_2 &\stackrel{2^{d+\alpha}}{\lesssim}& \frac{a^\alpha}{|x-y|^{d+\alpha}}\int_0^{t} g_{d,\beta_1}(t-s,x,z)\sqrt{s} ds\notag\\
    &=& \frac{a^\alpha t}{|x-y|^{d+\alpha}} \int_0^{t} \frac{\sqrt{s}\sqrt{t-s}}{t}(t-s)^{-(d+1)/2}\exp\left(-\frac{\beta_1|x-z|^2}{t-s}\right) ds\notag\\
    &\le& \frac{a^\alpha t}{|x-y|^{d+\alpha}} N^{\beta_1}(t,x,z).
  \end{eqnarray}
  Thus we have by \eqref{equ_I2t}-\eqref{equ_I2xz}
  \begin{equation*}
    I_2 \stackrel{c_6(d,\alpha,\beta_1,M,T)}{\lesssim} \left(t^{-d/2}\wedge \frac{a^\alpha t}{|x-y|^{d+\alpha}}\right) \left(N^{\beta_1}(t,x,z) + H^{(1+\alpha)/2}(t,z,y) \right).
  \end{equation*}
 Similarly, we have
  \begin{equation*}
    I_3 \stackrel{c_7(d,\alpha,\beta_2,M,T)}{\lesssim} \left(t^{-d/2}\wedge \frac{a^\alpha t}{|x-y|^{d+\alpha}}\right) \left(H^{1+\alpha/2}(t,x,z) + N^{\beta_2}(t,z,y)\right).
  \end{equation*}
  It remains to estimate $I_4$. If $|x-z|^2 \ge t-s$ and $|z-y|^2 \ge s$, then $ |x-z|\vee |z-y| \ge   \sqrt{t/2}$. Since $|x-z|\vee|z-y| \ge |x-y|/2$, we have $|x-z|\vee|z-y| \ge \frac{1}{2}(\sqrt{t}\vee|x-y|)$. Therefore
  \begin{eqnarray*}
    &&\frac{t-s}{|x-z|^{d+\alpha}}\frac{s}{|z-y|^{d+1+\alpha}}\\
    &=& \frac{1}{(|x-z|\vee|z-y|)^{d+\alpha}} \frac{1}{(|x-z|\wedge |z-y|)^{d+\alpha}} \frac{(t-s)s}{|z-y|}\\
    &\stackrel{2^{d+\alpha}}{\lesssim}& \frac{1}{(\sqrt{t}\vee|x-y|)^{d+\alpha}} \frac{(t-s)s}{(|x-z|\wedge |z-y|)^{d+1+\alpha}} \\
    &\le& \left(t^{-d/2+1-\alpha/2}\wedge \frac{t}{|x-y|^{d+\alpha}}\right) \frac{(t-s)s}{t} \left(\frac{1}{|x-z|^{d+1+\alpha}}+\frac{1}{|z-y|^{d+1+\alpha}}\right)\\
    &\stackrel{(T\vee 1)^{1-\alpha/2}}{\lesssim}& \left(t^{-d/2}\wedge \frac{t}{|x-y|^{d+\alpha}}\right) \left(\frac{t-s}{|x-z|^{d+1+\alpha}}+\frac{s}{|z-y|^{d+1+\alpha}}\right)\\
  \end{eqnarray*}
  Thus,
  \begin{eqnarray}\label{equ_I4_1}
    I_4 &=& a^{2\alpha} \int_0^t {\bf 1}_{\{|x-z|^2 \ge t-s\}} {\bf 1}_{\{|z-y|^2 \ge s\}}\frac{t-s}{|x-z|^{d+\alpha}}\frac{s}{|z-y|^{d+1+\alpha}}ds\notag\\
    &\stackrel{c_8(d,\alpha,T)}{\lesssim}& a^{2\alpha}\left(t^{-d/2}\wedge \frac{t}{|x-y|^{d+\alpha}}\right) \int_0^t {\bf 1}_{\{|x-z|^2 \ge t-s\}} {\bf 1}_{\{|z-y|^2 \ge s\}}\notag\\
    && \times \left(\frac{t-s}{|x-z|^{d+1+\alpha}}+\frac{s}{|z-y|^{d+1+\alpha}}\right)ds\notag\\
    &\stackrel{M^\alpha(M\vee 1)^\alpha}{\lesssim}& \left(t^{-d/2}\wedge \frac{a^{\alpha}t}{|x-y|^{d+\alpha}}\right) \int_0^t {\bf 1}_{\{|x-z|^2 \ge t-s\}} {\bf 1}_{\{|z-y|^2 \ge s\}} \notag\\
    && \times \left(\frac{t-s}{|x-z|^{d+1+\alpha}}+\frac{s}{|z-y|^{d+1+\alpha}}\right)ds.
  \end{eqnarray}
  Notice that
  \begin{eqnarray}\label{equ_I4_2}
    \int_0^t \frac{t-s}{|x-z|^{d+1+\alpha}}{\bf 1}_{\{|x-z|^2 \ge t-s\}} ds
     &=& \frac1{|x-z|^{d+1+\alpha}} \int_0^{t  \wedge |x-z|^2} r \, dr \notag\\
    &\le& 2^{-1} t^{1-\alpha/2} \left(\frac{1}{|x-z|^{d-1}} \wedge \frac{t^{1+\alpha/2}}{|x-z|^{d+1+\alpha}}\right)\notag\\
    &\le& 2^{-1}T^{1-\alpha/2} H^{1+\alpha/2}(t,x,z).
  \end{eqnarray}
  Similarly,
  \begin{equation}\label{equ_I4_3}
    \int_0^t \frac{s}{|z-y|^{d+1+\alpha}}{\bf 1}_{\{|z-y|^2 \ge s\}} ds \le 2^{-1}T^{1-\alpha/2} H^{1+\alpha/2}(t,z,y).
  \end{equation}
  We have by  (\ref{equ_I4_1}), (\ref{equ_I4_2}) and (\ref{equ_I4_3}),
  \begin{equation*}
    I_4 \le c_9(d,\alpha,M,T) \left(t^{-d/2}\wedge \frac{a^{\alpha}t}{|x-y|^{d+\alpha}}\right)\left(H^{1+\alpha/2}(t,x,z)+H^{1+\alpha/2}(t,z,y)\right).
  \end{equation*}
  Hence by (\ref{equ_Nupper}) and the fact that $\beta \mapsto H^\beta(t,x,y)$ is decreasing, we have
  \begin{eqnarray*}
    I &\stackrel{C_{\CTtherepI}(d,\alpha,\beta_1,\beta_2,M,T)}{\lesssim}& \left(H^{\gamma}(t,x,z)+H^{\gamma}(t,z,y)\right)\left(g_{d,\beta_1}(t,x,y) + t^{-d/2}\wedge \frac{a^\alpha t}{|x-y|^{d+\alpha}}\right).
  \end{eqnarray*}
 This completes the proof of (\ref{equ_3q}). Multiplying the both sides of (\ref{equ_3q}) by $|b(z)|$, we get
  \begin{eqnarray*}
    && \int_0^t\int_{\mathbb{R}^d} q_{d,\beta_1}^a(t-s,x,z)|b(z)|q_{d+1,\beta_2}^a(s,z,y) dzds\\
    &\stackrel{C_{\CTtherepI}}{\lesssim}& \int_{\mathbb{R}^d} q_{d,\beta_1}^a(t,x,y) |b(z)|\left(H^{\gamma}(t,x,z) + H^{\gamma}(t,z,y)\right) dz\\
    &\stackrel{2}{\lesssim}& q_{d,\beta_1}^a(t,x,y) \sup_{x\in \mathbb{R}^d} H_b^{\gamma}(t,x)
     \stackrel{C_{\CTHM}}{\lesssim}  M_b(\sqrt{t})q_{d,\beta_1}^a(t,x,y).
  \end{eqnarray*}
 This proves the lemma with $C_{\CTtherepII} = 2C_{\CTtherepI}C_{\CTHM}$.
\end{proof}

For $t > 0$ and $x,y \in \mathbb{R}^d$, we define
\begin{align*}
  |p|^{a,b}_0(t,x,y) &= p^a(t,x,y),\notag \\
  |p|_{k}^{a,b}(t,x,y) &= \int_0^t\int_{\mathbb{R}^d} |p|_{k-1}^{a,b}(t-s,x,z)|b(z)| |\nabla_z p^a(s,z,y)| dz ds,~ \text{ for } k \ge 1.
\end{align*}
For every $M > 0$ and $T > 0$, we can verify by induction that
\begin{equation}\label{equ_abspk}
  |p|_k^{a,b}(t,x,y) \le C_{\CTpaUc}(C_{\CTnablapaU}C_{\CTtherepII}M_b(\sqrt{t}))^k q_{d,C_{\CTpaUexp}/2}^a(t,x,y), ~~a \in (0,M], (t,x,y) \in (0,T]\times\mathbb{R}^d\times\mathbb{R}^d.
\end{equation}
Indeed, (\ref{equ_abspk}) holds for $k=0$. Assume (\ref{equ_abspk}) holds for $k$. Then by assumption and (\ref{equ_3p}),
\begin{eqnarray*}
  |p|_{k+1}^{a,b}(t,x,y)
  &\le& C_{\CTpaUc}(C_{\CTnablapaU}C_{\CTtherepII}M_b(\sqrt{t}))^k C_{\CTnablapaU}\int_0^t \int_{\mathbb{R}^d} q_{d,C_{\CTpaUexp}/2}^a(t-s,x,z)|b(z)|q_{d+1,3C_{\CTpaUexp}/4}^a(s,z,y) dz ds\\
  &\le& C_{\CTpaUc}(C_{\CTnablapaU}C_{\CTtherepII}M_b(\sqrt{t}))^k C_{\CTnablapaU}C_{\CTtherepII} M_b(\sqrt{t}) q_{d,C_{\CTpaUexp}/2}^{a}(t,x,y)\\
  &\le& C_{\CTpaUc}(C_{\CTnablapaU}C_{\CTtherepII}M_b(\sqrt{t}))^{k+1}q_{d,C_{\CTpaUexp}/2}^a(t,x,y).
\end{eqnarray*}
Thus for every $k \ge 1$, $t \in (0,T]$, $p_k^{a,b}(t, x, y)$ of \eqref{equ_pndef} is well defined and has bound
\begin{equation}\label{equ_pkupbd}
  |p_k^{a,b}(t,x,y)| \le |p|_k^{a,b}(t,x,y) \le C_{\CTpaUc}(C_{\CTnablapaU}C_{\CTtherepII}M_b(\sqrt{t}))^k q_{d,C_{\CTpaUexp}/2}^a(t,x,y) < \infty.
\end{equation}

\begin{lem}\label{thm_pkcts}
  Suppose $M > 0$. For every $a \in (0,M]$ and $k\ge 0$, $p_k^{a,b}(t,x,y)$ is jointly continuous on $(0,\infty)\times \mathbb{R}^d\times \mathbb{R}^d$.
\end{lem}
\begin{proof}
  We will use induction in $k$ to prove this lemma. Obviously,
  $ p_0^{a,b}(t,x,y)=p^a(t, x, y)$
   is jointly continuous on $(0,\infty)\times \mathbb{R}^d\times \mathbb{R}^d$. Assume $p_k^{a,b}(t,x,y)$ is jointly continuous. By (\ref{equ_Gaussianup}) with $\theta = 1$,
  \begin{equation}\label{equ_qup}
    q_{d,C_{\CTpaUexp}/2}^a(t,x,y) \le t^{-d/2}\wedge \frac{c_1 t}{|x-y|^{d+2}} + t^{-d/2} \wedge \frac{a^\alpha t}{|x-y|^{d+\alpha}}, \quad t > 0, x,y\in \mathbb{R}^d,
  \end{equation}
  for some positive constant $c_1$ depending only on $d$. Suppose $T > 1$ and $0<\varepsilon < 1/(2T)$. For $t \in [T^{-1},T]$ and $s\in [\varepsilon, t-\varepsilon]$, we have by (\ref{equ_pkupbd}) and (\ref{equ_qup}) that there is a constant $c_2 = c_2(d,\alpha,M,T,b)$ such that
  \begin{align}
    |p_k^{a,b}(t-s,x,z)| \le&~ C_{\CTpaUc}(C_{\CTnablapaU}C_{\CTtherepII} M_b(\sqrt{t}))^k q_{d,C_{\CTpaUexp}/2}^a(t-s,x,z) \le 2c_2 (t-s)^{-d/2} \le 2c_2\varepsilon^{-d/2},\label{equ_pkup}\\
    |\nabla_z p^a(s,z,y)| \le&~ C_{\CTnablapaU}q_{d+1,3C_{\CTpaUexp}/4}^a(s,z,y) \le 2C_{\CTnablapaU} s^{-(d+1)/2} \le 2C_{\CTnablapaU} \varepsilon^{-(d+1)/2}\label{equ_Gpup}, \\
    |p_k^{a,b}(t-s,x,z)| \le&~ 2c_2\frac{1}{|x-z|^{d+\alpha}}, \quad \text{if } |x-z| \ge 1. \nonumber
  \end{align}
  Then for $R\geq 1$,
  \begin{align*}
  & \sup_{x\in \R^d} \int_\varepsilon^{t-\varepsilon} \int_{|x-z|\ge R} |p_k^{a,b}(t-s,x,z)||b(z)||\nabla_z p^a(s,z,y)| dz ds  \\
  &   \le \sup_{x\in \R^d} 4c_2C_{\CTnablapaU} \varepsilon^{-(d+1)/2} T \int_{|x-z|\ge R} \frac{|b(z)|}{|x-z|^{d+\alpha}} dz,
  \end{align*}
  which goes to zero as $R \rightarrow \infty$. On the other hand, since $x \mapsto p^{a,b}_k(t-s,x,z)$ is continuous by assumption and $p^a(t,x,y)$ is smooth, we have for any $r > 0$, by (\ref{equ_pkup}-\ref{equ_Gpup}), the local integrability of $b$ and the dominated convergence theorem,
  \begin{equation*}
    (x,y) \mapsto \int_\varepsilon^{t-\varepsilon} \int_{|x-z|< R} p_k^{a,b}(t-s,x,z)b(z)\nabla_z p^a(s,z,y) dz ds
  \end{equation*}
  is continuous on $B(0,r)\times B(x,r)$. Thus, we can conclude that
  \begin{equation}\label{equ_equ1}
  (t,x,y) \mapsto \int_{\varepsilon}^{t-\varepsilon} \int_{\mathbb{R}^d} |p_k^{a,b}(t-s,x,z)||b(z)||\nabla_z p^a(s,z,y)| dz ds
  \end{equation}
  is jointly continuous on $[T^{-1},T]\times B(0,r)\times B(0,r)$. Since $r$ is arbitrary, (\ref{equ_equ1}) is jointly continuous on $[T^{-1},T]\times \mathbb{R}^d\times \mathbb{R}^d$. On the other hand, by (\ref{equ_Gpup}) and (\ref{equ_pkupbd}),
  \begin{align*}
    &\sup_{t\in [1/T,T]}\sup_{x,y\in\mathbb{R}^d} \int_{t-\varepsilon}^t \int_{\mathbb{R}^d} |p_k^{a,b}(t-s,x,z)||b(z)||\nabla_z p^a(s,z,y)| dz ds\\
    \le& \sup_{t\in [1/T,T]}\sup_{x\in\mathbb{R}^d} \sup_{s \in [t-\varepsilon,t], z,y\in \mathbb{R}^d}|\nabla_z p^a(s,z,y)| \int_{t-\varepsilon}^t \int_{\mathbb{R}^d} |p_k^{a,b}(t-s,x,z)||b(z)| dz ds\\
    \le& \sup_{s \in [1/(2T),T], z,y\in \mathbb{R}^d}|\nabla_z p^a(s,z,y)| \sup_{x\in\mathbb{R}^d} \int_{0}^\varepsilon \int_{\mathbb{R}^d} |p_k^{a,b}(s,x,z)||b(z)| dz ds\\
    \le& 2C_{\CTnablapaU}(2T)^{(d+1)/2} C_{\CTpaUc}(C_{\CTnablapaU}C_{\CTtherepII}M_b(\sqrt{\varepsilon}))^k\\
    &\qquad\cdot \sup_{x\in \mathbb{R}^d} \bigg(\int_{0}^\varepsilon \int_{\mathbb{R}^d}\!\!\!\! \sqrt{\varepsilon}|b(z)| \frac{g_{d,C_{\CTpaUexp}/2}(s,x,z)}{s^{1/2}} dzds +\!\! \int_{0}^\varepsilon\!\! \int_{|x-z|^2 \ge s} \!\!\!|b(z)| \frac{a^\alpha s}{|x-z|^{d+\alpha}} dz ds\bigg)\\
    \le& 2C_{\CTnablapaU}(2T)^{(d+1)/2}C_{\CTpaUc}(C_{\CTnablapaU}C_{\CTtherepII}M_b(\sqrt{\varepsilon}))^k\\
    &\qquad \cdot \bigg(\sqrt{\varepsilon} \sup_{x\in \mathbb{R}^d} \int_{\mathbb{R}^d} |b(z)| N^{C_{\CTpaUexp}/2}(\varepsilon,x,z) dz +\sup_{x\in \mathbb{R}^d} \int_{\mathbb{R}^d} a^\alpha |b(z)|\frac{\varepsilon^2\wedge |x-z|^4}{|x-z|^{d+\alpha}} dz \bigg)\\
    \le& 2C_{\CTnablapaU}(2T)^{(d+1)/2}C_{\CTpaUc}(C_{\CTnablapaU}C_{\CTtherepII}M_b(\sqrt{\varepsilon}))^k \left(c_3C_{\CTHM}\sqrt{\varepsilon} M_b(\sqrt{\varepsilon}) + a^\alpha \varepsilon^{(3-\alpha)/2} C_{\CTHM}M_b(\sqrt{\varepsilon})\right),
  \end{align*}
  which goes to zero as $\varepsilon \rightarrow 0$. Similarly, by (\ref{equ_pkup}),
  \begin{align*}
    &\sup_{t\in [1/T,T]}\sup_{x,y\in\mathbb{R}^d} \int_0^\varepsilon \int_{\mathbb{R}^d} |p_k^{a,b}(t-s,x,z)||b(z)||\nabla_z p^a(s,z,y)| dz ds\\
    \le& 2c_2(2T)^{d/2}C_{\CTpaUc}(C_{\CTnablapaU}C_{\CTtherepII}M_b(\sqrt{\varepsilon}))^k \left(c_3C_{\CTHM} M_b(\sqrt{\varepsilon}) + a^\alpha \varepsilon^{1-\alpha/2}
    C_{\CTHM}M_b(\sqrt{\varepsilon})\right),
  \end{align*}
  which goes to zero as $\varepsilon \rightarrow 0$. Therefore,
  \begin{equation*}
    p_{k+1}^{a,b}(t,x,y) = \int_{0}^{t} \int_{\mathbb{R}^d} p_k^{a,b}(t-s,x,z) b(z) \nabla_z p^a(s,z,y) dz ds
  \end{equation*}
  is jointly continuous on $[T^{-1},T]\times \mathbb{R}^d\times \mathbb{R}^d$
  for every $T>0$. This completes the proof.
\end{proof}

\const{\CTpnSum}
\begin{lem}\label{thm_pabbounds}
  Suppose $M > 0$. There are two positive constants $t_*(d,\alpha,M,b) > 0$ depending on $b$ only via the rate at which $M_b(r)$ goes to zero and $C_{\CTpnSum} = C_{\CTpnSum}(d,\alpha,M)>0$ such that for all $t \in (0,t_*]$ and $x,y\in \mathbb{R}^d$,
  \begin{align}\label{equ_pabupbd}
    \left|\sum_{k=0}^{\infty} p_k^{a,b}(t,x,y)\right| \le \sum_{k=0}^{\infty} |p_k^{a,b}(t,x,y)| \le C_{\CTpnSum} q_{d,C_{\CTpaUexp}/2}^a(t,x,y).
  \end{align}
  Moreover,  for all $|x-y|^2 < t \leq t_*$,
  \begin{equation}\label{equ_pablow}
    \sum_{k=0}^{\infty} p_k^{a,b}(t,x,y) \ge C_{\CTpnSum}^{-1}t^{-d/2}.
  \end{equation}
\end{lem}
\begin{proof}
  Let $C_{\CTpaUc}, C_{\CTnablapaU}, C_{\CTtherepII}$ be the constants in (\ref{equ_pkupbd}) with $T = 1$ and $C_{\CTpaLc}, C_{\CTpaLexp}$ be the constants in Theorem \ref{thm_Gradpabound} with $T = 1$. Since $b \in \mathbb{K}_{d,1}$, there is a constant $0< t_* < 1$ such that for all $t \in (0,t_*]$
  \begin{equation*}
    C_{\CTnablapaU}C_{\CTtherepII}M_b(\sqrt{t}) \le \frac 1 2\wedge \frac{C_{\CTpaLc}e^{-C_{\CTpaLexp}}}{8C_{\CTpaUc}},
  \end{equation*}
  and so by (\ref{equ_pkupbd}) with $T = 1$,
  \begin{align}\label{equ_pk1up}
    \sum_{k=1}^{\infty} |p_k^{a,b}(t,x,y)| \le& C_{\CTpaUc} \frac{C_{\CTnablapaU}C_{\CTtherepII}M_b(\sqrt{t})}{1-C_{\CTnablapaU}C_{\CTtherepII}M_b(\sqrt{t})} q_{d,C_{\CTpaUexp}/2}^a(t,x,y)\notag\\
    \le& 2C_{\CTpaUc} C_{\CTnablapaU}C_{\CTtherepII}M_b(\sqrt{t}) q_{d,C_{\CTpaUexp}/2}^a(t,x,y),\quad x,y\in\mathbb{R}^d.
  \end{align}
  Thus, by Theorem \ref{thm_pabound} with $T = 1$ and (\ref{equ_pk1up}), we have for all $(t,x,y)\in (0,t_*]\times\mathbb{R}^d\times\mathbb{R}^d$,
  \begin{equation*}
    \left|\sum_{k=0}^{\infty} p_k^{a,b}(t,x,y)\right| \le \sum_{k=0}^{\infty} |p_k^{a,b}(t,x,y)| \le 2C_{\CTpaUc}q_{d,C_{\CTpaUexp}/2}^a(t,x,y),
  \end{equation*}
  which gives (\ref{equ_pabupbd}). On the other hand, if $|x-y|^2 < t \le t_*$, then
  \begin{equation*}
    p^a(t,x,y) \ge C_{\CTpaLc}e^{-C_{\CTpaLexp}} t^{-d/2}\text{ and } q_{d,C_{\CTpaUexp}/2}^a(t,x,y) \le 2t^{-d/2}.
  \end{equation*}
  Thus, by (\ref{equ_pk1up}) again, we have for $ (t,x,y)\in (0,t_*]\times\mathbb{R}^d\times\mathbb{R}^d$ with $|x-y|^2 \le t$,
  \begin{align*}
    \sum_{k=0}^{\infty} p_k^{a,b}(t,x,y) \ge& p^a(t,x,y) - \sum_{k=1}^{\infty} |p_k^{a,b}(t,x,y)| \ge C_{\CTpaLc}e^{-C_{\CTpaLexp}} t^{-d/2} - \frac{C_{\CTpaLc}e^{-C_{\CTpaLexp}}}{2} t^{-d/2} \\
    =& \frac{C_{\CTpaLc}e^{-C_{\CTpaLexp}}}{2} t^{-d/2} .
     \end{align*}
\end{proof}

In the remainder of this paper, we fix $t_*$. By Lemma \ref{thm_pabbounds}, the series $\sum_{k = 0}^{\infty} p_k^{a,b}(t,x,y)$ absolutely converges on $(0,t_*]\times\mathbb{R}^d\times \mathbb{R}^d$. For every $a \in (0,M]$, define
\begin{equation}\label{equ_pabsmalldef}
  p^{a,b}(t,x,y) = \sum_{k = 0}^{\infty} p_k^{a,b}(t,x,y),\quad 0<t\le t_*\text{ and } x,y \in \mathbb{R}^d.
\end{equation}

\begin{lem}\label{thm_pabctssmall}
  Suppose $M > 0$. For every $a\in (0,M]$, $p^{a,b}(t,x,y)$ is jointly continuous on $(0,t_*]\times\mathbb{R}^d\times \mathbb{R}^d$.
\end{lem}
\begin{proof}
  For any $0 < t_1 < t_*$, we have
  \begin{equation*}
    \sup_{[t_1,t_*]\times\mathbb{R}^d\times \mathbb{R}^d} q_{d,C_{\CTpaUexp}/2}^a(t,x,y) \le 2t_1^{-d/2}< \infty.
  \end{equation*}
  By Lemma \ref{thm_pabbounds} and inequality (\ref{equ_pkupbd}), the series $\sum_{k = 0}^{\infty} p_k^{a,b}(t,x,y)$ converges uniformly on $[t_1,t_*]\times\mathbb{R}^d\times \mathbb{R}^d$. Since $t_1$ is arbitrary, the result follows from Lemma \ref{thm_pkcts}.
\end{proof}

\begin{thm}\label{thm_psmallsmgrp}
  Suppose $M > 0$. For every $a\in (0,M]$, $0 < s,t \le t_*$ with $s+t \le t_*$ and $x,y \in \mathbb{R}^d$, we have
  \begin{equation}\label{equ_CKsmall}
    p^{a,b}(t+s,x,y) = \int_{\mathbb{R}^d} p^{a,b}(t,x,z)p^{a,b}(s,z,y)dz.
  \end{equation}
\end{thm}
\begin{proof}
  Note that for $s,t > 0$ with $s+t\le t_*$,
  \begin{align*}
    p^{a,b}(t,x,z)p^{a,b}(s,z,y) &= \left(\sum_{m=0}^{\infty} p^{a,b}_m(t,x,z)\right) \left(\sum_{k=0}^{\infty} p^{a,b}_k(s,z,y)\right)\\
    &= \sum_{k=0}^{\infty} \sum_{m=0}^{k} p^{a,b}_m(t,x,z) p^{a,b}_{k-m}(s,z,y).
  \end{align*}
  So it suffices to prove that for any $k\ge 0$,
  \begin{equation}\label{equ_pksmgrp}
    p_k^{a,b}(t+s,x,y) = \sum_{m=0}^{k} \int_{\mathbb{R}^d} p_m^{a,b}(t,x,z) p_{k-m}^{a,b}(s,z,y) dz,
  \end{equation}
 which will be done inductively. When $k=0$, (\ref{equ_pksmgrp}) is clearly true since $p^{a,b}_0(t,x,y) = p^a(t,x,y)$. Suppose (\ref{equ_pksmgrp}) holds when $k = l$ and we have
  \begin{align*}
    p^{a,b}_{l+1}(t+s,x,y) =& \int_0^{t+s}\int_{\mathbb{R}^d} p^{a,b}_l(t+s-\tau,x,w)b(w)\nabla_w p^{a,b}_0(\tau,w,y)dw d\tau\\
    =& \left( \int_0^s +   \int_s^{t+s}   \right) \int_{\mathbb{R}^d} p^{a,b}_l(t+s-\tau,x,w)b(w)\nabla_w p^{a,b}_0(\tau,w,y) dwd\tau \\
    =& \int_0^s \int_{\mathbb{R}^d} \sum_{m=0}^{l} \int_{\mathbb{R}^d} p^{a,b}_m(t,x,z)p^{a,b}_{l-m}(s-\tau,z,w)dzb(w) \nabla_w p^{a,b}_0(\tau,w,y)dw d\tau\\
    &+ \int_s^{t+s}\!\!\!\int_{\mathbb{R}^d} p^{a,b}_l(t+s-\tau,x,w)b(w) \nabla_w\!\!\!\int_{\mathbb{R}^d} p^{a,b}_0(\tau-s,w,z) p^{a,b}_0(s,z,y)dz dw d\tau\\
    =& \sum_{m=0}^{l} \int_{\mathbb{R}^d} p^{a,b}_m(t,x,z)\int_0^s \int_{\mathbb{R}^d} p^{a,b}_{l-m}(s-\tau,z,w)b(w) \nabla_w p^{a,b}_0(\tau,w,y)dw d\tau dz\\
    &+ \int_{\mathbb{R}^d} \int_0^{t}\int_{\mathbb{R}^d} p^{a,b}_l(t-\tau,x,w)b(w) \nabla_w p^{a,b}_0(\tau,w,z)dw d\tau p^{a,b}_0(s,z,y)dz \\
    =& \sum_{m=0}^{l} \int_{\mathbb{R}^d} p^{a,b}_m(t,x,z) p^{a,b}_{l+1-m}(s,z,y) dz + \int_{\mathbb{R}^d} p^{a,b}_{l+1}(t,x,z)p^{a,b}_0(s,z,y)dz ,
  \end{align*}
where in the second to the last equality, we used Fubini's theorem since
 for  $(t,x,y)\in (0,t_*]\times\mathbb{R}^d\times\mathbb{R}^d$ and any $m,l\in \mathbb{Z}_+$, by (\ref{equ_abspk}) and Lemma \ref{thm_ace},
  \begin{align*}
    &\int_0^s \int_{\mathbb{R}^d} \int_{\mathbb{R}^d} |p^{a,b}_m(t,x,w)||p^{a,b}_l(s-\tau,w,z)||b(z)||\nabla_zp^a(\tau,z,y)dwdzd\tau \notag\\
    =& \int_{\mathbb{R}^d} |p^{a,b}_m(t,x,w)| \int_0^s \int_{\mathbb{R}^d}|p^{a,b}_l(s-\tau,w,z)||b(z)||\nabla_zp^a(\tau,z,y)dzd\tau dw\notag\\
    \le& \int_{\mathbb{R}^d} |p^{a,b}_m(t,x,w)||p^{a,b}|_{l+1}(s,w,y)dw < \infty.
  \end{align*}
  We have also used the fact that due to Theorem \ref{thm_pabound} and the dominated convergence theorem,
   $ \nabla_zp^a(\tau,z,y) = \int_{\mathbb{R}^d} \nabla_zp^a(\tau-s,z,w) p^a(s,w,y)dw$.
\end{proof}

In view of Theorem  \ref{thm_psmallsmgrp}, the definition of $p^{a,b}(t,x,y)$ can be uniquely extended to all $t > 0$ so that \eqref{e:1.8} holds for all $s, t>0$. Suppose $p^{a,b}(t,x,y)$ has been well defined on $(0, kt_*]\times \mathbb{R}^d\times \mathbb{R}^d$ for integer $k\geq  0$ and
\eqref{e:1.8} holds for all $s, t>0$ with $s+t\leq kt_*$.
For $t \in (kt_*, (k+1) t_*]$, we define
\begin{equation}\label{equ_pablargedef}
  p^{a,b}(t,x,y) = \int_{\mathbb{R}^d} p^{a,b}(kt_*,x,z)p^{a,b} (t-kt_*,z,y)dz, \quad x,y\in \mathbb{R}^d.
\end{equation}
One can verify easily that the Chapman-Kolmogorov equation \eqref{e:1.8}
holds for every $t, s >0$ with $t+s \leq (k+1)t_*$.
This proves that \eqref{e:1.8} holds for all $t, s>0$.

\begin{thm}\label{thm_pabctscsvtv}
  Suppose $M > 0$. For every $a \in (0,M]$, $p^{a,b}(t,x,y)$ is continuous on $(0,\infty)\times \mathbb{R}^d \times \mathbb{R}^d$ and $\int_{\mathbb{R}^d} p^{a,b}(t,x,y) dy = 1$ for every $t> 0$ and $x\in \mathbb{R}^d$.
\end{thm}
\begin{proof}
  The continuity of $p^{a,b}(t,x,y)$ for all $t > 0$ follows from Lemma \ref{thm_pabctssmall}, (\ref{equ_pablargedef}) and the dominated convergence theorem.

  It follows from (\ref{equ_ghk}) that $\int_{\mathbb{R}^d} \nabla_x p^a(t,x,y) dy =0$ for all $t > 0$ and $x\in \mathbb{R}^d$. Thus for every $k \ge 1$, by Lemma \ref{thm_ace}, \eqref{equ_pndef}, (\ref{equ_pkupbd}) and Fubini's theorem,
  \begin{align*}
    \int_{\mathbb{R}^d} p_k^{a,b}(t,x,y) dy &= \int_{\mathbb{R}^d} \int_{\mathbb{R}^d} \int_{0}^t p_{k-1}^{a,b}(t-s,x,z) b(z) \cdot \nabla_z p^a(s,z,y) dsdzdy\\
    &= \int_{\mathbb{R}^d} \int_0^t p_{k-1}^{a,b}(t-s,x,z) b(z) \cdot \int_{\mathbb{R}^d} \nabla_z p^a(s,z,y) dy ds dz = 0.
  \end{align*}
   In view of (\ref{equ_pkupbd}) and the dominated convergence theorem, we have for all $t\in (0,t_*]$ and $x \in \mathbb{R}^d$,
  \begin{equation*}
    \int_{\mathbb{R}^d} p^{a,b}(t,x,y)dy = \sum_{k=0}^{\infty} \int_{\mathbb{R}^d} p_k^{a,b}(t,x,y)dy = \int_{\mathbb{R}^d} p^a(t,x,y)dy = 1,
  \end{equation*}
  which extends to all $t > 0$ by (\ref{equ_pablargedef}).
\end{proof}

For bounded measurable function $f$ on $\mathbb{R}^d$, $t > 0$ and $x\in \mathbb{R}^d$, we define  operator $P^{a,b}_t$
\begin{equation*}
  P^{a,b}_t f(x) = \int_{\mathbb{R}^d} p^{a,b}(t,x,y) f(y) dy.
\end{equation*}
It follows from \eqref{equ_CKsmall} that $P^{a,b}_sP^{a,b}_t = P^{a,b}_{t+s}$.

The following theorem tells us that the generator of $\{P_t^{a,b},t\ge 0\}$ is $ {\mathcal{L}}^{a,b}$ in the weak sense. The proof is almost the same to part of the proof of \cite[Theorem 1]{BogdanJakubowski.2007}. We give the details of the proof for completeness.
For any compact set $K \subset \mathbb{R}^d$ and $r > 0$, let $K^r = \{y\in \mathbb{R}^d: \exists x\in K \text{ such that } |x-y| < r\}$ be the $r$-neighborhood of $K$.
\begin{thm}\label{thm_generator}
  Suppose $M > 0$. For every $a \in (0,M]$ and  for all $f\in C_c^{\infty}(\mathbb{R}^d)$, $g \in C_\infty(\mathbb{R}^d)$,
  \begin{equation*}
    \lim_{t\rightarrow 0}\int_{\mathbb{R}^d} \frac{P^{a,b}_tf(x) - f(x)}{t} g(x) dx = \int_{\mathbb{R}^d} \mathcal{L}^{a,b}f(x) g(x) dx.
  \end{equation*}
\end{thm}
\begin{proof}
  Note that for all $t \in (0,t_*]$,
  \begin{align*}
    \int_{\mathbb{R}^d} \frac{P^{a,b}_tf(x) - f(x)}{t} g(x) dx =& \frac{1}{t}\int_{\mathbb{R}^d}\left(\int_{\mathbb{R}^d}p^{a,b}_0(t,x,y) f(y) dy - f(x)\right) g(x) dx \\
    &+ \frac{1}{t}\int_{\mathbb{R}^d}\int_{\mathbb{R}^d} \left(p^{a,b}_1(t,x,y) + \sum_{k = 2}^{\infty}p^{a,b}_k(t,x,y) \right) f(y) g(x) dy dx.
  \end{align*}
  Since $p^{a,b}_0(t,x,y) = p^a(t,x,y)$ we have
  \begin{equation*}
    \lim_{t \rightarrow 0} \frac{1}{t}\int_{\mathbb{R}^d}\left(\int_{\mathbb{R}^d}p^{a,b}_0(t,x,y) f(y) dy - f(x)\right) g(x) dx = \int_{\mathbb{R}^d}\left(\Delta + a^\alpha \Delta^{\alpha/2}\right)f(x)g(x) dx.
  \end{equation*}
  For $t \in (0,t_*]$, let $I(t) = t^{-1}\int_{\mathbb{R}^d}\int_{\mathbb{R}^d} p^{a,b}_1(t,x,y) f(y) g(x) dx$.
  We claim that $I(t)$ converges to
  $\int_{\mathbb{R}^d} ( b(x)\cdot \nabla f(x)) g(x) dx$
  as $t \rightarrow 0$. By (\ref{equ_pndef}), Fubini's theorem and integration by parts,
  \begin{equation*}
    I(t) = \int_{\mathbb{R}^d}\int_{\mathbb{R}^d}\int_{\mathbb{R}^d}\int_0^t \frac{1}{t} p^a(t-s,x,z)p^a(s,z,y) ds ( b(z)\cdot \nabla f(y)) g(x) dz dy dx.
  \end{equation*}
 Since $ g(x) \nabla f(y)$ is uniformly continuous and bounded, for every $\varepsilon > 0$, there is $\delta > 0$ so that
   $|g(x) \nabla f(y) - g(z) \nabla f(w)| < \varepsilon$ for
   $|x-z|<\delta$ and $|y-w|<\delta$.
   Let   $M_0 = \sup_{x,y\in \mathbb{R}^d} | g(x) \nabla f(y) | $,
  and $K$ be the support of $\nabla f$.
    Recall that $K^1$ denotes the $1$-neighborhood of $K$. Clearly
  \begin{align*}
    &|I(t) - \int_{\mathbb{R}^d} ( b(z)\cdot \nabla f(z)) g(z) dz|\\
    \le& \int_{\mathbb{R}^d}\int_{\mathbb{R}^d}\int_{\mathbb{R}^d} \int_0^t \frac{1}{t} p^a(t-s,x,z)p^a(s,z,y) ds|b(z)|\,
    |g(x) \nabla f(y) - g(z) \nabla f(z)|dx dy dz \\
    =& \left(\int_{(K^1)^c}\int_{\mathbb{R}^d\times  \mathbb{R}^d}\int_0^t + \int_{K^1}\int_{(B(z,\delta)\times B(z,\delta))^c}\int_0^t + \int_{K^1}\int_{B(z,\delta)\times B(z,\delta)}\int_0^t \right) \cdots ds dx dy dz\\
      =&: J_1 + J_2 + J_3.
  \end{align*}
  We  estimate $J_1,J_2$ and $J_3$ separately. Note that if $x \in K$ and $z \in (K^1)^c$, then $|x - z| \ge 1$ and so by Theorem \ref{thm_pabound}, for $x,y\in \mathbb{R}^d$ and $0<s < t$,
  \begin{equation*}
    p^a(t-s,x,z) \le C_{\CTpaUc}q_{d,C_{\CTpaUexp}}^a(t-s,x,z) \le C_{\CTpaUc}c_1\frac{t-s}{|x-z|^{d+\alpha}}.
  \end{equation*}
  where $c_1$ is a positive constant depending only on $d,\alpha,M$. Thus,
  \begin{align*}
    J_1
 &= \int_{(K^1)^c }\int_{K\times \mathbb{R}^d}\int_0^t \frac{1}{t} p^a(t-s,x,z)p^a(s,z,y) ds|b(z)|\,  |g(x) \nabla f(y) |dx dy dz  \\
  &\leq
  M_0\int_{(K^1)^c}\int_{K}\int_0^t \left(\int_{\mathbb{R}^d} p^a(s,z,y) dy\right) \frac{1}{t} p^a(t-s,x,z) |b(z)|ds dx dz\\
    &\leq
   M_0C_{\CTpaUc}c_1\int_{(K^1)^c}\int_{K}\int_0^t \frac{1}{t} \frac{t-s}{|x - z|^{d+\alpha}} |b(z)|ds dx dz\\
    &\le tM_0C_{\CTpaUc}c_1  |K|\sup_{x \in \mathbb{R}^d} \int_{|x - z| \ge 1}\frac{|b(z)|}{|x - z|^{d+\alpha}} dz \\
    & \le tM_0C_{\CTpaUc}c_1 |K|\sup_{x \in \mathbb{R}^d} H_b^{(3-\alpha)/2}(1,x) ,
  \end{align*}
 which tends to zero as $t\to 0$.
 Similarly, if $(x,y) \in (B(z,\delta)\times B(z,\delta))^c$, then $|x - z| \ge \delta$ or $|y - z|\ge \delta$. Since $b$ is locally integrable,  we have
  \begin{equation*}
    J_2 \le 2tM_0C_{\CTpaUc}c_1 \int_{K^1}\int_{|x - z| \ge \delta} |b(z)|\frac{1}{|x - z|^{d+\alpha}}dx dz \le 2tM_0C_{\CTpaUc}c_1\delta^{-d-2}\int_{K^1} |b(z)|dz \rightarrow 0
  \end{equation*}
  as $t \rightarrow 0$.
  \begin{equation*}
    J_3 \le \varepsilon \int_{K^1}\int_{\mathbb{R}^d}\int_{\mathbb{R}^d} \int_0^t \frac{1}{t} p^a(t-s,x,z)p^a(s,z,y) ds|b(z)|dx dy dz \le \varepsilon \int_{K^1} |b(z)| dz.
  \end{equation*}
  Since $\varepsilon$ is arbitrary, we have $\lim_{t\rightarrow 0} I(t) = \int_{\mathbb{R}^d} ( b(z) \cdot \nabla f(z)) g(z) dz$.

  By (\ref{equ_pndef}), Lemma \ref{thm_pabbounds} and the dominated convergence theorem, we have
  \begin{equation*}
    \sum_{k = 2}^{\infty} p^{a,b}_k(t,x,y) = \int_{\mathbb{R}^d} \int_0^t \left(\sum_{k = 1}^{\infty} p^{a,b}_k(t-s,x,z)\right) b(z)\cdot \nabla_z p^a(s,z,y) ds dz.
  \end{equation*}
  Similar  to the estimate of $I(t)$, by Fubini's theorem, integration by parts and (\ref{equ_pk1up}), we have for all $t \in (0,t_*]$
  \begin{align*}
    &\frac{1}{t}\left|\int_{\mathbb{R}^d}\int_{\mathbb{R}^d} \left(\sum_{k = 2}^{\infty}p^{a,b}_k(t,x,y) \right) f(y) g(x) dy dx\right|\\
    =& \frac{1}{t}\left|\int_{\mathbb{R}^d}\int_{\mathbb{R}^d} \int_{\mathbb{R}^d} \int_0^t \left(\sum_{k = 1}^{\infty} p^{a,b}_k(t-s,x,z)\right) p^a(s,z,y) (b(z) \cdot \nabla f(y)) g(x) ds dz dy dx\right|\\
    \le&\frac{2C_{\CTpaUc}C_{\CTnablapaU}C_{\CTtherepII}M_b(\sqrt{t})}{t}\int_{\mathbb{R}^d}\int_{\mathbb{R}^d} \int_{\mathbb{R}^d} \int_0^t q_{d,C_{\CTpaUexp}}^a(t-s,x,z) p^a(s,z,y) |b(z)| |g(x)\nabla f(y)| ds dz dy dx\\
    \le&2C_{\CTpaUc}C_{\CTnablapaU}C_{\CTtherepII}M_b(\sqrt{t})\left(\frac{2C_{\CTpaLexp}}{C_{\CTpaUexp}}\right)^{d/2}C_{\CTpaLc}^{-1}\\ &\times\frac{1}{t}\int_{\mathbb{R}^d}\int_{\mathbb{R}^d} \int_{\mathbb{R}^d} \int_0^t p^a(\frac{2C_{\CTpaLexp}}{C_{\CTpaUexp}}(t-s),x,z) p^a(s,z,y) |b(z)| |g(x)\nabla f(y)| ds dz dy dx,
  \end{align*}
  which goes to zero as $t \rightarrow 0$. This completes the proof.
\end{proof}

\section{Uniqueness and Positivity}
\const{\CTpabAbsUcp}
\const{\CTpabAbsUexp}
\const{\CTpabAbsUcq}

\begin{thm}\label{thm_pababsup}
  Suppose $M > 0$. There are constants $C_{\CTpabAbsUcp} = C_{\CTpabAbsUcp}(d,\alpha,M)$, $C_{\CTpabAbsUexp} = C_{\CTpabAbsUexp}(d,\alpha,M,b)$ such that for all $a \in (0,M]$,
  \begin{equation}\label{equ_pabAbsUp}
    \left|p^{a,b}(t,x,y) \right| \le C_{\CTpabAbsUcp} e^{C_{\CTpabAbsUexp}t} p^a({2C_{\CTpaLexp}t}/{C_{\CTpaUexp}},x,y), \quad t>0 \text{ and } x,y \in \mathbb{R}^d.
  \end{equation}
  Consequently, for any $T>0$, there is a constant $C_{\CTpabAbsUcq}=C_{\CTpabAbsUcq}(d,\alpha,M,T)$ such that
  \begin{equation*}
    \left|p^{a,b}(t,x,y) \right| \le C_{\CTpabAbsUcq} e^{C_{\CTpabAbsUexp}t} q_{d,C_{\CTpaUexp}^2/(2C_{\CTpaLexp})}^a(t,x,y),\quad t\in(0,T] \text{ and } x,y \in \mathbb{R}^d.
  \end{equation*}
\end{thm}
\begin{proof}
  Note that by the expression of $q_{d,C_{\CTpaUexp}/2}^a(t,x,y)$ and the lower bound of $p^a(t,x,y)$ in Theorem \ref{thm_pabound} with $T=1$,
  \begin{equation}\label{equ_qpup}
    q_{d,C_{\CTpaUexp}/2}^a(t,x,y) \le \left(\frac{2C_{\CTpaLexp}}{C_{\CTpaUexp}}\right)^{d/2} q_{d,C_{\CTpaLexp}}^a({2C_{\CTpaLexp}t}/{C_{\CTpaUexp}},x,y) \le \left(\frac{2C_{\CTpaLexp}}{C_{\CTpaUexp}}\right)^{d/2} C_{\CTpaLc}^{-1}p^a({2C_{\CTpaLexp}t}/{C_{\CTpaUexp}},x,y).
  \end{equation}
  Recall that $t_*$ is the constant in Lemma \ref{thm_pabbounds}. If $t < t_*$, by (\ref{equ_pabupbd}) and Theorem \ref{thm_pabound},
  \begin{equation*}
    \left|p^{a,b}(t,x,y)\right| \le C_{\CTpnSum}q_{d,C_{\CTpaUexp}/2}^a(t,x,y) \le c_1p^a({2C_{\CTpaLexp}t}/{C_{\CTpaUexp}},x,y),
  \end{equation*}
  where $c_1 = \frac{C_{\CTpnSum}(2C_{\CTpaLexp})^{d/2}}{C_{\CTpaLc}C_{\CTpaUexp}^{d/2}}$ depends only on $d,\alpha,M$. It remains to consider the case $t > t_*$. Let $k=\lfloor t/t_*\rfloor+1$, then $t/k \in (0,t_*)$. Combining (\ref{equ_qpup}), (\ref{e:1.8}) and (\ref{equ_pabupbd}), we have
  \begin{align*}
    \left|p^{a,b}(t,x,y)\right| &\le \int_{\mathbb{R}^{d(k-1)}} c_1^k p^a(\frac{2C_{\CTpaLexp}}{C_{\CTpaUexp}}\frac t k,x,x_1) \cdots p^a(\frac{2C_{\CTpaLexp}}{C_{\CTpaUexp}}\frac t k,x_{k-1},y) dx_1\cdots dx_{k-1}\\
    &= c_1^k p^a({2C_{\CTpaLexp}t}/{C_{\CTpaUexp}},x,y)\\
    &\le c_1 c_1^{\frac{t}{t_*}}p^a({2C_{\CTpaLexp}t}/{C_{\CTpaUexp}},x,y),
  \end{align*}
  which gives the first conclusion with $C_{\CTpabAbsUcp}=c_1$ and $C_{\CTpabAbsUexp} = \frac{1}{t_*}\ln c_1$. Furthermore, by the upper bound of $p^a(t,x,y)$ in Theorem \ref{thm_pabound}, for $t \in (0,T]$ and $x,y\in \mathbb{R}^d$
  \begin{equation*}
    p^a({2C_{\CTpaLexp}t}/{C_{\CTpaUexp}},x,y) \le C_{\CTpaUc}q_{d,C_{\CTpaUexp}/2}^a({2C_{\CTpaLexp}t}/{C_{\CTpaUexp}},x,y) \le \frac{2C_{\CTpaLexp}C_{\CTpaUc}}{C_{\CTpaUexp}} q_{d,C_{\CTpaUexp}^2/(2C_{\CTpaLexp})}^a(t,x,y).
  \end{equation*}
  Combining the last two displays, we finish the proof by setting $C_{\CTpabAbsUcq}= c_1((2C_{\CTpaLexp}C_{\CTpaUc}/C_{\CTpaUexp})\vee 1)$.
\end{proof}

\begin{thm}
  Suppose that $M > 0$ and $b \in \mathbb{K}_{d,1}$. For every $a \in (0,M]$, $p^{a,b}(t,x,y)$ satisfies (\ref{equ_Duhamel}) for all $t > 0$ and $x,y\in\mathbb{R}^d$.
\end{thm}
\begin{proof}
  Recall that $t^*$ is the constant in Lemma \ref{thm_pabbounds}. We first prove that $p^{a,b}(t,x,y)$ satisfies (\ref{equ_Duhamel}) for all $t \in (0,t^*]$ and $x,y\in\mathbb{R}^d$. Indeed, by (\ref{equ_pabsmalldef}), (\ref{equ_pk1up}), Theorem \ref{thm_Gradpabound}, (\ref{equ_3p}) and the dominated convergence theorem, we have for all $t \in (0,t^*]$,
  \begin{align*}
    p^{a,b}(t,x,y) &= \sum_{n=0}^\infty p^{a,b}_n(t,x,y)\\
    &=
     p^a (t, x, y)
     + \sum_{n=1}^\infty \int_0^t\int_{\mathbb{R}^d} p^{a,b}_{n-1}(t-s,x,z) b(z) \nabla_z p^a(s,z,y) dzds\\
    &=
     p^a (t, x, y)
     + \int_0^t\int_{\mathbb{R}^d} \sum_{n=1}^\infty p^{a,b}_{n-1}(t-s,x,z) b(z) \nabla_z p^a(s,z,y) dzds\\
    &=
     p^a (t, x, y)
     + \int_0^t\int_{\mathbb{R}^d} p^{a,b}(t-s,x,z) b(z) \nabla_z p^a(s,z,y) dzds.
  \end{align*}
  Now, we use induction in $k$ to prove (\ref{equ_Duhamel}) for all $t > 0$. Suppose that (\ref{equ_Duhamel}) is true for $t \in (0,2^kt^*](k\ge 0)$ and for all $x,y\in \mathbb{R}^d$. We will prove (\ref{equ_Duhamel}) is true for $t \in (2^kt^*,2^{k+1}t^*]$. Setting $s = t/2 \in (2^{k-1}t^*,2^kt^*]$, by (\ref{e:1.8}), Theorem \ref{thm_pababsup}, (\ref{equ_3p}) and Fubini's theorem, we have
  \begin{align*}
    p^{a,b}(t,x,y) &= \int_{\mathbb{R}^d} p^{a,b}(s,x,z) p^{a,b}(s,z,y) dz\\
    &= \int_{\mathbb{R}^d} p^{a,b}(s,x,z) \left(
    p^{a}(s,z,y)
    + \int_0^s\int_{\mathbb{R}^d} p^{a,b}(s-r,z,w)b(w)\nabla_w p^a(r,w,y) dwdr\right)dz\\
    &= \int_{\mathbb{R}^d}
    p^{a}(s,x,z) p^{a}(s,z,y) dz\\
    &\quad + \int_{\mathbb{R}^d}\left(\int_0^s\int_{\mathbb{R}^d} p^{a,b}(s-r,x,u)b(u)\nabla_up^a(r,u,z) dudr\right) p^a(s,z,y) dz\\
    &\quad + \int_{\mathbb{R}^d} p^{a,b}(s,x,z) \left(\int_0^s\int_{\mathbb{R}^d} p^{a,b}(s-r,z,w)b(w)\nabla_wp^a(r,w,y) dwdr\right)dz\\
    &=
     p^{a}(t,x,y)
     +\int_0^s\int_{\mathbb{R}^d} p^{a,b}(s-r,x,u)b(u)\left(\int_{\mathbb{R}^d}\nabla_u p^a(r,u,z) p^a(s,z,y) dz\right) dudr\\
    &\quad + \int_0^s\int_{\mathbb{R}^d} \left(\int_{\mathbb{R}^d} p^{a,b}(s,x,z) p^{a,b}(s-r,z,w)dz\right) b(w)\nabla_w p^a(r,w,y) dwdr\\
    &=
     p^{a}(t,x,y)
     +\int_0^s\int_{\mathbb{R}^d} p^{a,b}(s-r,x,u)b(u) \nabla_u p^a(r+s,u,y) dudr\\
    &\quad + \int_0^s\int_{\mathbb{R}^d} p^{a,b}(2s-r,x,w) b(w)\nabla_wp^a(r,w,y) dwdr\\
    &=
     p^{a}(t,x,y)
     +\int_s^{2s}\int_{\mathbb{R}^d} p^{a,b}(2s-r,x,u)b(u) \nabla_u p^a(r,u,y) dudr\\
    &\quad + \int_0^s\int_{\mathbb{R}^d} p^{a,b}(2s-r,x,w) b(w)\nabla_wp^a(r,w,y) dwdr\\
    &=
     p^{a}(t,x,y)
     +\int_0^t\int_{\mathbb{R}^d} p^{a,b}(t-r,x,z)b(z) \nabla_z p^a(r,z,y) dzdr,
  \end{align*}
  where in the forth equality, we can change the order of integral and $\nabla$, since for any $t_1,t_2\in(0,\infty)$ and $x,y\in \mathbb{R}^d$,
  \begin{equation*}
    \nabla_x p^a(t_1+t_2,x,y) = \int_{\mathbb{R}^d} \nabla_x p^a(t_1,x,z)p^a(t_2,z,y) dz,
  \end{equation*}
  which can be proved by Theorem \ref{thm_Gradpabound} and the dominated convergence theorem.
\end{proof}

\begin{thm}\label{thm_hkUnique}
  Suppose that $M > 0$ and $b \in \mathbb{K}_{d,1}$. For every $a \in (0,M]$, $p^{a,b}(t,x,y)$ is the unique continuous heat kernel that satisfies the Chapman-Kolmogorov equation \eqref{e:1.8} on $(0,\infty)\times\mathbb{R}^d\times\mathbb{R}^d$, Duhamel's formula \eqref{equ_Duhamel}
on $(0,t_0]\times\mathbb{R}^d\times\mathbb{R}^d$ for some constant $t_0>0$ and that for some $c_1>0$,
  \begin{equation}\label{equ_qbcUpSmallTime}
    \left|p^{a,b}(t,x,y)\right| \le c_1 p^a(t,x,y)　\quad \hbox{for } t\in (0,  t_0] \hbox{ and } x, y\in \R^d.
  \end{equation}
\end{thm}

\begin{proof}
  Suppose that $\overline{p}(t,x,y)$ is any continuous heat kernel that satisfies Duhamel's formula (\ref{equ_Duhamel}) and (\ref{equ_qbcUpSmallTime}) for $(t,x,y) \in (0,t_0]\times\mathbb{R}^d\times\mathbb{R}^d$. Without loss of generality, we may and do assume that $t_0 \le t^*$. Firstly, let $R_1(t,x,y) = \int_0^t\int_{\mathbb{R}^d} \overline{p}(t-s,x,z)b(z)\nabla_z p^a(s,z,y) dz ds$ and
  \begin{equation*}
    R_n(t,x,y) = \int_0^t\int_{\mathbb{R}^d} R_{n-1}(t-s,x,z)b(z)\nabla_z p^a(s,z,y) dz ds,\quad n\ge 2.
  \end{equation*}
  Similar to the arguments that lead to (\ref{equ_pkupbd}), by (\ref{equ_qbcUpSmallTime}), we can recursively verify that $R_n(t,x,y)$ is well defined. Furthermore, we have the upper bound of $\left|R_n(t,x,y)\right|$:
  \begin{equation*}
    \left|R_n(t,x,y)\right| \le  c_1(C_{\CTnablapaU}C_{\CTtherepII}M_b(\sqrt{t}))^n q_{d,C_{\CTpaUexp}/2}^a(t,x,y).
  \end{equation*}
  On the other hand, using Duhamel's formula (\ref{equ_Duhamel}) inductively, we have for every $n \ge 1$,
  \begin{equation*}
    \overline{p}(t,x,y) = \sum_{j=0}^{n-1} p^{a,b}_j(t,x,y) + R_n(t,x,y),
  \end{equation*}
  where $p^{a,b}_j(t,x,y)$ is defined by (\ref{equ_pndef}). Note that for all $(t,x,y) \in (0,t_0]\times\mathbb{R}^d\times\mathbb{R}^d$, by the proof of Lemma \ref{thm_pabbounds}, $C_{\CTnablapaU}C_{\CTtherepII}M_b(\sqrt{t})\le 1/2$ and so
  \begin{equation*}
    \left|R_n(t,x,y)\right| \le c_1 2^{-n} q_{d,C_{\CTpaUexp}/2}^a(t,x,y) < \infty,
  \end{equation*}
  which goes to zero as $n \rightarrow \infty$. Thus, we have
  \begin{equation*}
    \overline{p}(t,x,y) = \sum_{k=0}^\infty p^{a,b}_k(t,x,y) = p^{a,b}(t,x,y),\quad \text{ for all }(t,x,y) \in (0,t_0]\times\mathbb{R}^d\times\mathbb{R}^d.
  \end{equation*}
  Since both $\overline{p}$ and $p^{a,b}$ satisfy the Chapman-Kolmogorov equation (\ref{e:1.8}) on $(0,\infty)\times\mathbb{R}^d\times\mathbb{R}^d$, we have $\overline{p} = p^{a,b}$ on $(0,\infty)\times\mathbb{R}^d\times\mathbb{R}^d$.
\end{proof}

Unlike that  in  \cite{BogdanJakubowski.2007},
 it is not easy to show the positivity of $p^{a, b}(t, x, y)$ directly from its construction.
We show $p^{a,b}(t, x, y)\geq 0$ by adopting the approach from  \cite{ChenWang.2012},
using Hille-Yosida-Ray theorem when $b$ is bounded and continuous and  then using approximation for general $b$.
\begin{lem}\label{thm_feller}
  Suppose $M > 0$. For every $a\in (0,M]$ and every $t > 0$, $P_t^{a,b}$ maps bounded functions to continuous functions. Furthermore, $\{P_t^{a,b},t\ge 0\}$ is a strongly continuous semigroup in $C_{\infty}(\mathbb{R}^d)$.
\end{lem}
\begin{proof}
  By Theorem \ref{thm_pababsup} and Theorem \ref{thm_pabctscsvtv}, one can easily verify that $P_t^{a,b}$ maps bounded functions to continuous functions for every $t > 0$. For every $f \in C_{\infty}(\mathbb{R}^d)$ and $t>0$, by Theorem \ref{thm_pababsup},
  \begin{align*}
    \lim_{|x|\rightarrow \infty}\left|P^{a,b}_tf(x)\right|\le& \lim_{|x|\rightarrow \infty} \int_{\mathbb{R}^d} C_{\CTpabAbsUcq}e^{C_{\CTpabAbsUexp}t}q_{d,C_{\CTpaUexp}^2/(2C_{\CTpaLexp})}^a(t,x,y) f(y) dy\\
    \le& \lim_{|x|\rightarrow \infty} \int_{\mathbb{R}^d} C_{\CTpabAbsUcq}e^{C_{\CTpabAbsUexp}t}q_{d,C_{\CTpaUexp}^2/(2C_{\CTpaLexp})}^a(t,0,y) f(x+y) dy = 0,
  \end{align*}
  which shows $P^{a,b}_t f \in C_\infty(\mathbb{R}^d)$. Moreover, since $f$ is uniformly continuous on $\mathbb{R}^d$, for every $\varepsilon > 0$, there is a constant $\delta > 0$ such that $|f(x) - f(y)| \le \varepsilon$ for all $x,y\in \mathbb{R}^d$ with $|x-y| \le \delta$. And so by (\ref{equ_qup}),
  \begin{align*}
    &\lim_{t\rightarrow 0} \sup_{s\le t} \sup_{x\in \mathbb{R}^d} \int_{|x-y|\ge \delta} |p^{a,b}(s,x,y)| dy\\
    \le& \lim_{t\rightarrow 0} \sup_{s\le t} \sup_{x\in \mathbb{R}^d} \int_{|x-y|\ge \delta} C_{\CTpabAbsUcq}e^{C_{\CTpabAbsUexp}s} q_{d,C_{\CTpaUexp}^2/(2C_{\CTpaLexp})}^a(s,x,y) dy\\
    \le& \lim_{t\rightarrow 0}\sup_{x\in \mathbb{R}^d} \int_{|x-y|\ge \delta} C_{\CTpabAbsUcq}e^{C_{\CTpabAbsUexp}t} c_1\left(\frac{t}{|x-y|^{d+2}} + \frac{t}{|x-y|^{d+\alpha}}\right) dy = 0,
  \end{align*}
  where $c_1$ is some positive constant depending only on $d,\alpha,M$. Thus, we have
  \begin{align*}
    &\lim_{t\rightarrow 0} \sup_{s\le t} \sup_{x\in \mathbb{R}^d} |P^{a,b}_sf(x) - f(x)| = \lim_{t\rightarrow 0} \sup_{s\le t} \sup_{x\in \mathbb{R}^d} \left|\int_{\mathbb{R}^d} p^{a,b}(s,x,y) f(y)dy - f(x)\right| \\
    \le& \lim_{t\rightarrow 0} \sup_{s\le t} \sup_{x\in \mathbb{R}^d} \int_{|x-y| < \delta} |p^{a,b}(s,x,y)||f(x)-f(y)|dy\\
    \le& \lim_{t\rightarrow 0} \sup_{s\le t} \sup_{x\in \mathbb{R}^d} \int_{|x-y| < \delta} C_{\CTpabAbsUcp}e^{C_{\CTpabAbsUexp}t} p^a( {2C_{\CTpaLexp}^2s}/{C_{\CTpaUexp}^2} ,x,y) |f(x)-f(y)| dy\\
    \le& \varepsilon C_{\CTpabAbsUcp},
  \end{align*}
  which shows that $\lim_{t\rightarrow 0} \|P^{a,b}_t f - f\|_{\infty} = 0$.
\end{proof}

\begin{lem}\label{thm_pab'pos}
  Suppose $M > 0$ and the function $b$ is bounded and continuous on $\mathbb{R}^d$. Then, for every $a \in (0,M]$,
  \begin{equation*}
    p^{a,b}(t,x,y) \ge 0,\qquad t > 0 \text{ and } x,y\in \mathbb{R}^d.
  \end{equation*}
\end{lem}

\begin{proof}
  Denote the Feller generator of $\{P^{a,b}_t, t\ge 0\}$ in $C_\infty(\mathbb{R}^d)$ by $\widehat{\mathcal{L}}^{a,b}$, which is a closed operator. For every $f \in C_c^2(\mathbb{R}^d)$, since $b$ is continuous, it is easy to see that
  $ {\mathcal{L}}^{a,b}f \in C_\infty(\mathbb{R}^d)$. Similar to Theorem \ref{thm_generator}, we claim that $(P^{a,b}_t f - f)/t$ uniformly converges to ${\mathcal{L}}^{a,b}f$ as $t \rightarrow 0$. Indeed, for any $t \in (0,t_*]$,
  \begin{align*}
    &\|(P^{a,b}_tf -f)/t - \mathcal{L}^{a,b} f\|_\infty \\
    =& \sup_{x\in\mathbb{R}^d} \left|\frac 1 t \left(\int_{\mathbb{R}^d} \sum_{k=0}^{\infty} p^{a,b}_k(t,x,y) f(y)dy - f(x)\right) - \left(\Delta + a^\alpha \Delta^{\alpha/2} + b\cdot\nabla\right)f(x) \right|\\
    \le& \sup_{x\in\mathbb{R}^d} \left|\frac{1}{t} \left(\int_{\mathbb{R}^d} p^{a,b}_0(t,x,y) f(y)dy - f(x)\right) - \left(\Delta + a^\alpha \Delta^{\alpha/2}\right)f(x) \right|\\
    &+ \sup_{x\in\mathbb{R}^d} \left|\frac{1}{t} \int_{\mathbb{R}^d} p^{a,b}_1(t,x,y) f(y)dy - b(x)\nabla f(x)\right|\\
    &+ \sup_{x\in\mathbb{R}^d} \bigg|\frac{1}{t} \int_{\mathbb{R}^d} \sum_{k=2}^{\infty} p^{a,b}_k(t,x,y) f(y)dy\bigg|\\
    =& I_1+I_2+I_3.
  \end{align*}
  It follows that $I_1$ goes to zero as $t \rightarrow 0$ since $\Delta + a^\alpha \Delta^{\alpha/2}$ is the generator of $Z^a$. We next treat $I_2$ as we did with $I$ in the proof of Theorem \ref{thm_generator}. Let $M_0 = \sup_{x\in\mathbb{R}^d}|b(x) \nabla f(x)|$. Since $f \in C_c^2(\mathbb{R}^d)$, for any $\varepsilon > 0$, there is constant $\delta > 0$ such that $|b(z)\nabla f(y) - b(x)\nabla f(x)| < \varepsilon$ for all $x\in \mathbb{R}^d$ and $(z,y) \in B(x,\delta) \times B(x,\delta)$.
 On the other hand, for  $(z,y) \in (B(x,\delta) \times B(x,\delta))^c$, either  $|z-x| \ge \delta$ or $|z-y| \ge \delta$.
 Hence by (\ref{equ_qup}),
   \begin{align*}
    I_2 \le& \sup_{x\in \mathbb{R}^d} \int_0^t \int_{\mathbb{R}^d}\int_{\mathbb{R}^d} \frac 1 t p^a(t-s,x,z) p^a(s,z,y) \left|b(z)\nabla_yf(y) - b(x)\nabla_xf(x)\right| dz dy ds\\
    \le& \sup_{x\in \mathbb{R}^d} \int_{B(x,\delta)\times B(x,\delta)} \int_0^t \frac 1 t p^a(t-s,x,z) p^a(s,z,y) |b(z)\nabla_yf(y) - b(x)\nabla_xf(x)| ds dz dy \\
    &+ \sup_{x\in \mathbb{R}^d}\int_{(B(x,\delta)\times B(x,\delta))^c} \int_0^t \cdots dsdzdy\\
    \le& \varepsilon + 2M_0\sup_{x\in \mathbb{R}^d}\int_{|x-z|\ge \delta}\int_0^t c_1\left(\frac{t-s}{|x-z|^{d+2}} + \frac{t-s}{|x-z|^{d+\alpha}}\right) ds dz\\
    \le& \varepsilon + M_0c_1\omega_d(\delta^{-2}/2+\delta^{\alpha}/\alpha) t.
  \end{align*}
  where $c_1$ is some positive constant depending only on $d,\alpha,M$. Since $\varepsilon$ is arbitrary, $I_2$ goes to zero as $t \rightarrow 0$. Similar to $I_2$, we can prove that $I_3$ goes to zero as $t \rightarrow 0$. Thus we have
  \begin{equation}\label{equ_firstcon}
    C_c^2(\mathbb{R}^d) \subset D(\widehat{\mathcal{L}}^{a,b}) \text{ and } \widehat{\mathcal{L}}^{a,b}f = \mathcal{L}^{a,b}f \quad \text{for all } f\in C_c^2(\mathbb{R}^d).
  \end{equation}
  On the other hand, for $\lambda > C_{\CTpabAbsUexp}$, by Theorem \ref{thm_pababsup},
  \begin{equation}
    \begin{split}
      \label{equ_rsvnt} \sup_{x\in \mathbb{R}^d} \int_0^{\infty} e^{-\lambda t} |P^{a,b}_tf(x)| dt \le& \sup_{x\in \mathbb{R}^d} \int_0^{\infty} e^{-\lambda t} \int_{\mathbb{R}^d} |p^{a,b}(t,x,y)| |f(y)| dy dt\\
      \le& \|f\|_\infty \int_0^\infty C_{\CTpabAbsUcp}e^{-(\lambda - C_{\CTpabAbsUexp}) t} dt = c_\lambda \|f\|_\infty,
    \end{split}
  \end{equation}
  where $c_\lambda = C_{\CTpabAbsUcp}/(\lambda - C_{\CTpabAbsUexp})$.  Consider the strongly continuous semigroup $\{e^{-C_{\CTpabAbsUexp}t}P^{a,b}_t,t \ge 0\}$ with its generator $\widehat{\mathcal{L}}^{a,b}- C_{\CTpabAbsUexp}$. By (\ref{equ_rsvnt}), the resolvent set $\rho(\widehat{\mathcal{L}}^{a,b} - C_{\CTpabAbsUexp})$ of $\widehat{\mathcal{L}}^{a,b} - C_{\CTpabAbsUexp}$ contains $(0,\infty)$. Moreover, $\widehat{\mathcal{L}}^{a,b} - C_{\CTpabAbsUexp}$ satisfies the positive
  maximum principle in view of (\ref{equ_firstcon}) and \cite[Theorem 3.5.3]{Applebaum.2004}. Therefore, $\{e^{-C_{\CTpabAbsUexp}t}P^{a,b}_t,t \ge 0\}$ is a positivity preserving semigroup on $C_\infty(\mathbb{R}^d)$ by
  Hille-Yosida-Ray theorem (see \cite[Theorem 3.5.1]{Applebaum.2004}).
  Since $\{e^{-C_{\CTpabAbsUexp}t}P^{a,b}_t,t \ge 0\}$ has a continuous kernel $e^{-C_{\CTpabAbsUexp}t}p^{a,b}(t,x,y)$, we have $p^{a,b}(t,x,y) \ge 0$ for all $(t,x,y) \in (0,\infty)\times \mathbb{R}^d\times \mathbb{R}^d$.
\end{proof}

In the rest of this section, we show by an approximation argument that Lemma \ref{thm_pab'pos} continues to hold for $b\in \mathbb{K}_{d,1}$. Let $\varphi$ be a non-negative function in $C_c^{\infty}(\mathbb{R}^d)$ with $supp(\varphi) \subset B(0,1)$ and $\int_{\mathbb{R}^d} \varphi (x) dx = 1$. For $n \ge 1$, define $\varphi_n(x) := n^d\varphi(nx)$ and
\begin{equation*}
  b_n(x) = \int_{\mathbb{R}^d} \varphi_n(x-y) b(y) dy, \quad x\in \mathbb{R}^d.
\end{equation*}
For any compact set $K \subset \mathbb{R}^d$ and $r > 0$, recall that $K^r$ is the $r$-neighborhood of $K$. For any $0 \le r_1 \le r_2 \le +\infty$ and $\beta \ge 0$, we have
\begin{equation}\label{equ_bnb}
  \begin{split}
    \sup_{x\in K} \int_{|x-y|\in [r_1,r_2)} \frac{|b_n(y)|}{|x-y|^{d-1+2\beta}} dy \le& \sup_{x\in K} \int_{|x-y|\in [r_1,r_2)} \int_{\mathbb{R}^d} \frac{\varphi_n(y-z)|b(z)|}{|x-y|^{d-1+2\beta}} dz dy\\
    =& \sup_{x\in K} \int_{|x-y|\in [r_1,r_2)} \int_{|z| < 1/n} \frac{\varphi_n(z)|b(y-z)|}{|x-y|^{d-1+2\beta}} dz dy\\ =& \sup_{x\in K} \int_{|z| < 1/n} \varphi_n(z)\int_{|x-z-y|\in [r_1,r_2)}\frac{|b(y)|}{|x-z-y|^{d-1+2\beta}} dydz\\
    \le&\int_{|z| < 1/n} \varphi_n(z) \sup_{x\in K^1} \int_{|x-y|\in [r_1,r_2)}\frac{|b(y)|}{|x-y|^{d-1+2\beta}} dydz\\ =&\sup_{x\in K^1} \int_{|x-y|\in [r_1,r_2)} \frac{|b(y)|}{|x-y|^{d-1+2\beta}} dy.
  \end{split}
\end{equation}
In particular, for every $r > 0$ and $n \ge 1$, by setting $r_1 = 0, r_2 = r$ and $\beta = 0$, we have
\begin{equation}\label{equ_MbnMb}
  M_{b_n}(r) \le M_b(r).
\end{equation}

Recall that $\gamma = (1+\alpha\wedge 1)/2$.
\begin{lem}\label{thm_Hbbn}
  $H_{b-b_n}^{\gamma}(t,x)$ converges to 0 uniformly on compact subsets of $(0,+\infty)\times \mathbb{R}^d$ as $n \rightarrow \infty$.
\end{lem}
\begin{proof}
  Let $[t_0,T_0]\times K \subset (0,+\infty)\times \mathbb{R}^d$ be an arbitrary compact set. Then, we have
  \begin{align*}
    \sup_{(t,x)\in [t_0,T_0]\times K} H_{b-b_n}^{\gamma}(t,x) \le& \sup_{x\in K} \int_{\mathbb{R}^d} \left(\frac{1}{|x-y|^{d-1}}\wedge \frac{T_0^{\gamma}}{|x-y|^{d-1+2\gamma}}\right)|b(y)-b_n(y)| dy\\
 \le& \sup_{x\in K} \left(\int_{|x-y| < r} +\int_{r\le |x-y|< R} +\int_{|x-y| \ge R} \right) \cdots dy\\
    =&: I_1+I_2+I_3,
  \end{align*}
where $0 < r < R< \infty$.
Taking $r_1 = 0, r_2 = r$ and $\beta = 0$ in (\ref{equ_bnb}), we have
  \begin{equation*}
    I_1 \le 2 \sup_{x\in K^1}
 \int_{|x-y|  < r} \frac{|b(y)|}{|x-y|^{d-1}} dy.
 \end{equation*}
  Since $b \in \mathbb{K}_{d,1}$, for any $\varepsilon > 0$, we can choose $r$ small enough such that
 $    I_1 \le 2\cdot \frac{\varepsilon}{4} = \frac{\varepsilon}{2}$.
   Taking $r_1 = R$, $ r_2 = \infty$
and $\beta = \gamma$ in (\ref{equ_bnb}), we have
  \begin{equation*}
  I_3 \le 2T_0^{\gamma} \sup_{x\in K^1} \int_{|x-y|\ge R}
 \frac{|b(y)|}{|x-y|^{d-1+2\gamma}} dy.
  \end{equation*}
 Fix a point $x_0 \in K^1$ and take $R > {\rm diam} (K^1)$. Note that for $x\in K^1$ with $|x-y|  \ge R$, we have
  \begin{align*}
    &|x_0 - y| \ge |x -y| - |x-x_0| \ge R - {\rm diam} (K^1),\\
    &\frac{|x_0-y|}{|x-y|} \le \frac{|x_0-x|+|x-y|}{|x-y|} \le \frac{{\rm diam} (K^1)}{R}+1 \le 2,
  \end{align*}
  and so
  \begin{equation*}
    \begin{split}
      I_3 \le 2^{d+2}T_0^{\gamma} \int_{|x_0-y|  \ge R-{\rm diam} (K^1)} \frac{|b(y)|}{|x_0-y|^{d-1+2\gamma}} dy.\\
    \end{split}
  \end{equation*}

  By Lemma \ref{thm_HMb} and the dominated convergence theorem, we can choose $R$ large enough such that
 $ I_3 < \frac{\varepsilon}{2}$.
  Now, we fix the above $r,R$.
Let $R_1>0$ so that $K^1 \subset B(0, R_1)$. Since $b\in L^1_{loc}(\R^d)$,
 \begin{equation*}
     \varlimsup_{n \rightarrow \infty} I_2 \le \varlimsup_{n \rightarrow \infty}
 r^{-(d-1)/2} \int_{|y|  < R_1+R} |b(y) - b_n(y)| dy = 0.
 \end{equation*}
  Then, we have
  \begin{equation*}
    \varlimsup_{n \rightarrow \infty} \sup_{(t,x)\in [t_0,T_0]\times K} H_{b-b_n}^{\gamma}(t,x) \le \frac{\varepsilon}{2} + \frac{\varepsilon}{2} + 0.
  \end{equation*}
  This proves the lemma since $\varepsilon$ is arbitrary.
\end{proof}

\const{\CTbnUc}
\const{\CTbnUexp}
\begin{lem}\label{thm_pkbn_b_upbd}
  Suppose $M > 0$ and $T > 0$. There exist positive constants $C_{\CTbnUc}=C_{\CTbnUc}(d,\alpha,M,T)$ and $C_{\CTbnUexp}=C_{\CTbnUexp}(d,\alpha,M,T)$ so that for every $n \ge 1$, $j \ge 1$ and all $a \in (0,M]$, $(t,x,y)\in (0,T]\times \mathbb{R}^d\times \mathbb{R}^d$,
  \begin{equation}\label{equ_pkbnb}
    |p_j^{a,b_n}(t,x,y)-p_j^{a,b}(t,x,y)| \le C_{\CTbnUc}\left(C_{\CTbnUexp} M_b(\sqrt{t})\right)^{j-1}\left( (H_{b-b_n}^{\gamma}(t,x) + H_{b-b_n}^{\gamma}(t,y)\right) q_{d,C_{\CTpaUexp}/2}^a(t,x,y).
  \end{equation}
\end{lem}

\begin{proof}
  We prove (\ref{equ_pkbnb}) inductively in $j$. Since $\beta \mapsto q^a_{d,\beta}(t,x,y)$ is decreasing, by Theorem \ref{thm_pabound}, for all $0<s<t\le T$ and $x,z\in \mathbb{R}^d$,
  \begin{equation*}
    p^a(t-s,x,z) \le C_{\CTpaUc}q_{d,C_{\CTpaUexp}}^a(t-s,x,z) \le C_{\CTpaUc} q_{d,C_{\CTpaUexp}/2}^a(t-s,x,z).
  \end{equation*}
We have by (\ref{equ_pndef}), Theorem \ref{thm_Gradpabound} and (\ref{equ_3q}) with $\beta_1=C_{\CTpaUexp}/2$ and $\beta_2 = 3C_{\CTpaUexp}/4$,
  \begin{eqnarray*}
    &&|p_1^{a,b_n}(t,x,y) - p_1^{a,b}(t,x,y)|\\
    &=& \bigg|\int_0^t\int_{\mathbb{R}^d} p^a(t-s,x,z)(b(z)-b_n(z))\nabla_z p^a(s,z,y) dz ds\bigg|\\
    &\le& \int_0^t\int_{\mathbb{R}^d} p^a(t-s,x,z)|b(z)-b_n(z)||\nabla_z p^a(s,z,y)| dz ds\\
    &\stackrel{C_{\CTpaUc}C_{\CTnablapaU}}{\lesssim}& \int_0^t \int_{\mathbb{R}^d} q_{d,C_{\CTpaUexp}/2}^a(t-s,x,z)|b(z)-b_n(z)|q_{d+1,3C_{\CTpaUexp}/4}^a(s,z,y) dz ds\\
    &=& \int_{\mathbb{R}^d} |b(z)-b_n(z)|\bigg(\int_0^t q_{d,C_{\CTpaUexp}/2}^a(t-s,x,z)q_{d+1,3C_{\CTpaUexp}/4}^a(s,z,y)ds \bigg) dz \\
    &\stackrel{C_{\CTtherepI}}{\lesssim}& \bigg(\int_{\mathbb{R}^d}|b(z)-b_n(z)| \left(H^{\gamma}(t,x,z) + H^{\gamma}(t,z,y) \right) dz\bigg)\cdot q_{d,C_{\CTpaUexp}/2}^a(t,x,y)\\
    &=& \left(H_{b-b_n}^{\gamma}(t,x) +H_{b-b_n}^{\gamma}(t,y)\right)q_{d,C_{\CTpaUexp}/2}^a(t,x,y).
  \end{eqnarray*}
  This proves \eqref{equ_pkbnb} for $j=1$.
  Assume (\ref{equ_pkbnb}) is true for $j=k\geq 1$.
   By (\ref{equ_pndef}),
  \begin{align*}
    &|p_{k+1}^{a,b_n}(t,x,y) - p_{k+1}^{a,b}(t,x,y)|\\
    \le& \int_0^t\int_{\mathbb{R}^d} |p_k^{a,b_n}(t-s,x,z) - p_k^{a,b}(t-s,x,z)| |b_n(z)| |\nabla_z p^a(s,z,y)| dzds\\
    &+ \int_0^t\int_{\mathbb{R}^d} |p_k^{a,b}(t-s,x,z)| |b(z) - b_n(z)| |\nabla_z p^a(s,z,y)| dzds\\
    =&: I_1 + I_2.
  \end{align*}
  Let $C_{\CTbnUc} = C_{\CTpaUc}C_{\CTnablapaU}C_{\CTtherepI}$ and $C_{\CTbnUexp} = 2^{d+3}C_{\CTHM} C_{\CTnablapaU} C_{\CTtherepI}$. Then
  \begin{eqnarray*}
    I_1 &\stackrel{C_{\CTbnUc}C_{\CTnablapaU}}{\lesssim}& \int_0^t \int_{\mathbb{R}^d} \left(C_{\CTbnUexp} M_b(\sqrt{t-s})\right)^{k-1}\left(H_{b-b_n}^{\gamma}(t-s,x) + H_{b-b_n}^{\gamma}(t-s,z)\right)\\
    && \times ~ q_{d,C_{\CTpaUexp}/2}^a(t-s,x,z) |b_n(z)| q_{d+1,3C_{\CTpaUexp}/4}^a(s,z,y) dz ds\\
    &\le& \left(C_{\CTbnUexp} M_b(\sqrt{t})\right)^{k-1} \int_{\mathbb{R}^d} \left(H_{b-b_n}^{\gamma}(t,x) + H_{b-b_n}^{\gamma}(t,z)\right) |b_n(z)| \\
    && \times \left(\int_0^t q_{d,C_{\CTpaUexp}/2}^a(t-s,x,z) q_{d+1,3C_{\CTpaUexp}/4}^a(s,z,y)  ds\right)dz\\
    &\stackrel{C_{\CTtherepI}}{\lesssim}& \left(C_{\CTbnUexp} M_b(\sqrt{t})\right)^{k-1}\int_{\mathbb{R}^d} \left(H_{b-b_n}^{\gamma}(t,x) + H_{b-b_n}^{\gamma}(t,z)\right) |b_n(z)| \\
    && \times \left(H^{\gamma}(t,x,z) + H^{\gamma}(t,z,y)\right)q_{d,C_{\CTpaUexp}/2}^a(t,x,y) dz\\
    &\le& \left(C_{\CTbnUexp} M_b(\sqrt{t})\right)^{k-1} \Big[H_{b-b_n}^{\gamma}(t,x) \left(H_{b_n}^{\gamma}(t,x) + H_{b_n}^{\gamma}(t,y)\right)\\
    &&+ \int_{\mathbb{R}^d} H_{b-b_n}^{\gamma}(t,z) |b_n(z)| \left(H^{\gamma}(t,x,z) + H^{\gamma}(t,z,y)\right)dz \Big]q_{d,C_{\CTpaUexp}/2}^a(t,x,y) \\
    &\le& \left(C_{\CTbnUexp} M_b(\sqrt{t})\right)^{k-1} \Big[2C_{\CTHM} M_b(\sqrt{t}) H_{b-b_n}^{\gamma}(t,x) + \int_{\mathbb{R}^d} \int_{\mathbb{R}^d} |b(w) - b_n(w)| |b_n(z)|\\
    &&\times H^{\gamma}(t,z,w)\left(H^{\gamma}(t,x,z) + H^{\gamma}(t,z,y)\right)dz dw\Big]q_{d,C_{\CTpaUexp}/2}^a(t,x,y)
  \end{eqnarray*}
  Note that
  \begin{eqnarray*}
    &&H^{\gamma}(t,z,w) \wedge H^{\gamma}(t,x,z)\\
    &=& \left(\frac{1}{|w-z|^{d-1}} \wedge \frac{t^{\gamma}}{|w - z|^{d-1+2\gamma}}\right) \wedge \left(\frac{1}{|z-x|^{d-1}} \wedge \frac{t^{\gamma}}{|z - x|^{d-1+2\gamma}}\right)\\
    &\stackrel{2^{d+1}}{\lesssim}& \left(\frac{1}{|w-x|^{d-1}} \wedge \frac{t^{\gamma}}{|w-x|^{d-1+2\gamma}}\right) = H^{\gamma}(t,x,w).
  \end{eqnarray*}
  Similarly,
  \begin{equation*}
    H^{\gamma}(t,z,w) \wedge H^{\gamma}(t,y,z) \stackrel{2^{d+1}}{\lesssim} H^{\gamma}(t,y,w).
  \end{equation*}
  Thus
  \begin{eqnarray*}
    &&\int_{\mathbb{R}^d} \int_{\mathbb{R}^d} |b(w) - b_n(w)| |b_n(z)| H^{\gamma}(t,z,w)\left(H^{\gamma}(t,x,z) + H^{\gamma}(t,z,y)\right)dz dw\\
    &\stackrel{2^{d+1}}{\lesssim}& \int_{\mathbb{R}^d} \int_{\mathbb{R}^d} |b(w) - b_n(w)| |b_n(z)|\Big[ H^{\gamma}(t,x,w) \left(H^{\gamma}(t,z,w) + H^{\gamma}(t,x,z)\right)\\
    && + H^{\gamma}(t,y,w)\left(H^{\gamma}(t,z,w) + H^{\gamma}(t,y,z)\right)\Big] dz dw\\
    &=& \int_{\mathbb{R}^d}|b(w) - b_n(w)| \Big[H^{\gamma}(t,x,w)\left(H_{b_n}^{\gamma}(t,w) + H_{b_n}^{\gamma}(t,x)\right)\\
    && + H^{\gamma}(t,y,w)\left(H_{b_n}^{\gamma}(t,w) + H_{b_n}^{\gamma}(t,y)\right)\Big] dw\\
    &\stackrel{2C_{\CTHM}}{\lesssim}& M_b(\sqrt{t})\int_{\mathbb{R}^d}|b(w) - b_n(w)| \left( H^{\gamma}(t,x,w)+ H^{\gamma}(t,y,w) \right) dw\\
    &=& M_b(\sqrt{t}) \left(H_{b-b_n}^{\gamma}(t,x) + H_{b-b_n}^{\gamma}(t,y) \right).
  \end{eqnarray*}
  Therefore
  \begin{eqnarray*}
    I_1 &\stackrel{C_{\CTbnUc}}{\lesssim}&\left(C_{\CTbnUexp} M_b(\sqrt{t})\right)^{k-1} 2C_{\CTHM}C_{\CTnablapaU}C_{\CTtherepI} M_b(\sqrt{t}) \\
    &&\qquad \times \left[H_{b-b_n}^{\gamma}(t,x) + 2^{d+1}\left(H_{b-b_n}^{\gamma}(t,x) + H_{b-b_n}^{\gamma}(t,y) \right)\right] q_{d,C_{\CTpaUexp}/2}^a(t,x,y)\\
    &\le&\left(2C_{\CTHM} C_{\CTnablapaU} C_{\CTtherepI} M_b(\sqrt{t})\right)^{k} 2^{(d+2)(k-1)}(2^{d+1}+1) \left(H_{b-b_n}^{\gamma}(t,x) + H_{b-b_n}^{\gamma}(t,y) \right) q_{d,C_{\CTpaUexp}/2}^a(t,x,y).
  \end{eqnarray*}
  On the other hand, by (\ref{equ_pkupbd}),
  \begin{eqnarray*}
    I_2 &\stackrel{C_{\CTpaUc}C_{\CTnablapaU}}{\lesssim}& \left(2C_{\CTHM} C_{\CTnablapaU} C_{\CTtherepI} M_b(\sqrt{t})\right)^{k}\int_0^t \int_{\mathbb{R}^d} q_{d,C_{\CTpaUexp}/2}^a(t-s,x,z) |b(z) - b_n(z)| q_{d+1,3C_{\CTpaUexp}/4}^a(s,z,y) dz ds\\
    &\stackrel{C_{\CTtherepI}}{\lesssim}& \int_0^t \int_{\mathbb{R}^d} |b(z) - b_n(z)| \left(H^{\gamma}(t,x,z) + H^{\gamma}(t,x,z) \right) q_{d,C_{\CTpaUexp}/2}^a(t,x,y) dz\\
    &=& \left(H_{b-b_n}^{\gamma}(t,x) + H_{b-b_n}^{\gamma}(t,y) \right) q_{d,C_{\CTpaUexp}/2}^a(t,x,y).
  \end{eqnarray*}
  Thus
  \begin{eqnarray*}
    &&|p_{k+1}^{a,b_n}(t,x,y) - p_{k+1}^{a,b}(t,x,y)|\\
    &\stackrel{C_{\CTbnUc}}{\lesssim}& \left(2C_{\CTHM} C_{\CTnablapaU} C_{\CTtherepI} M_b(\sqrt{t})\right)^{k} \left(2^{(d+2)(k-1)}(2^{d+1}+2) \right) \left(H_{b-b_n}^{\gamma}(t,x) + H_{b-b_n}^{\gamma}(t,y) \right) q_{d,C_{\CTpaUexp}/2}^a(t,x,y)\\
    &\le& \left(C_{\CTbnUexp} M_b(\sqrt{t})\right)^{k} \left(H_{b-b_n}^{\gamma}(t,x) + H_{b-b_n}^{\gamma}(t,y) \right) q_{d,C_{\CTpaUexp}/2}^a(t,x,y).
  \end{eqnarray*}
  This completes the proof of the lemma.
\end{proof}

\begin{lem}\label{thm_pabncnvg}
  Suppose $M > 0$ and $T > 0$. For every $a\in (0,M]$, $0 < T_0 < T$ and compact set $K \subset \mathbb{R}^d$, $p^{a,b_n}(t,x,y)$ converges to $p^{a,b}(t,x,y)$ uniformly in $[T_0, T] \times K \times K$ as $n \rightarrow \infty$.
\end{lem}

\begin{proof}
 We divide the proof into two parts. In the first part, we show this lemma holds on $[T_0,T_1]\times K\times K$ for some $T_1 > 0$.
  In the second part,
we prove that for any $k \ge 1$, $p^{a,b_n}(t,x,y)$ converges to $p^{a,b}(t,x,y)$ uniformly in $[kT_1/2, (k+1)T_1/2] \times K \times K$ as $n \rightarrow \infty$.
Taking $k = 2\lfloor T/T_1\rfloor + 1$ then yields the claim of the lemma.

 (i)  By (\ref{equ_pabsmalldef}) and Lemma \ref{thm_pkbn_b_upbd} with $T = 1$, for all $t \in (0,t_*]$ and $x,y \in K$,
  \begin{eqnarray*}
    && |p^{a,b_n}(t,x,y) - p^{a,b}(t,x,y)|
    \le \sum_{k=1}^{\infty} |p_k^{a,b_n}(t,x,y) - p_k^{a,b}(t,x,y)|\\
    &\le& C_{\CTbnUc}\sum_{k=1}^{\infty} \left(C_{\CTbnUexp} M_b(\sqrt{t})\right)^{k} \left(H_{b-b_n}^{\gamma}(t,x) + H_{b-b_n}^{\gamma}(t,y) \right) q_{d,C_{\CTpaUexp}/2}^a(t,x,y).
  \end{eqnarray*}
  Since $b\in \mathbb{K}_{d,1}$, there is a constant $0<T_1<t_*$ so that
   $ C_{\CTbnUexp} M_b(\sqrt{T_1}) \le 1/2$.
  Then for all $t \le T_1$ and $x,y \in \mathbb{R}^d$,
  \begin{equation}\label{equ_pbn_b_upbd}
    \begin{split}
      |p^{a,b_n}(t,x,y) - p^{a,b}(t,x,y)| &\le C_{\CTbnUc} \sum_{k=1}^\infty 2^{-(k-1)}\left(H_{b-b_n}^{\gamma}(t,x) + H_{b-b_n}^{\gamma}(t,y)\right) q_{d,C_{\CTpaUexp}/2}^a(t,x,y) \\
      &\le 2C_{\CTbnUc} \left(H_{b-b_n}^{\gamma}(t,x) + H_{b-b_n}^{\gamma}(t,y)\right) q_{d,C_{\CTpaUexp}/2}^a(t,x,y).
    \end{split}
  \end{equation}
  Without loss of generality, we may and do assume $T_0 < T_1/2$. Note that $q_{d,C_{\CTpaUexp}/2}^a(t,x,y) \le 2T_0^{-d/2}$ for $T_0 \le t \le T_1$ and
   $x,y \in \mathbb{R}^d$. By (\ref{equ_pbn_b_upbd}) and Lemma \ref{thm_Hbbn},
  \begin{equation}\label{equ_T_1}
    \begin{split}
      &\limsup_{n\rightarrow \infty} \sup_{t\in [T_0,T_1]}\sup_{x,y \in K}|p^{a,b_n}(t,x,y) - p^{a,b}(t,x,y)|\\
      &\le 4T_0^{-d/2}C_{\CTbnUc} \limsup_{n\rightarrow \infty} \sup_{x,y \in K} \left(H_{b-b_n}^{\gamma}(T_1,x) + H_{b-b_n}^{\gamma}(T_1,y)\right)
       = 0.
    \end{split}
  \end{equation}

(ii)  We prove this part inductively in $k$. Indeed, it is true when $k=1$ by Step I. Assume that for any compact set $\widetilde{K}$ and $1\le k \le j$, $p^{a,b_n}(t,x,y)$ converges to $p^{a,b}(t,x,y)$ uniformly in $[kT_1/2, (k+1)T_1/2] \times \widetilde{K} \times \widetilde{K}$ as $n \rightarrow \infty$.

  Let $t_1 = T_1/2$. By Chapman-Kolmogorov equation (\ref{e:1.8}), For every $t \in [(j+1)T_1,(j+2)T_1/2]$,
  \begin{align*}
    |p^{a,b_n}(t,x,y) - p^{a,b}(t,x,y)| \le& \int_{\mathbb{R}^d} |p^{a,b_n}(t-t_1,x,z) - p^{a,b}(t-t_1,x,z)||p^{a,b_n}(t_1,z,y)|dz\\
    &\quad+ \int_{\mathbb{R}^d} |p^{a,b}(t-t_1,x,z)| |p^{a,b_n}(t_1,z,y) - p^{a,b}(t_1,z,y)|dz\\
    =&:I_1 + I_2.
  \end{align*}
  By Theorem \ref{thm_pababsup} and (\ref{equ_MbnMb}), for every $t \in [(j+1)T_1,(j+2)T_1/2]$,
  \begin{align*}
    &|p^{a,b_n}(t-t_1,x,z)| \le C_{\CTpabAbsUcq}e^{C_{\CTpabAbsUexp}(t-t_1)}q_{d,C_{\CTpaUexp}^2/(2C_{\CTpaLexp})}^a(t-t_1,x,z) \le 2C_{\CTpabAbsUcq}e^{C_{\CTpabAbsUexp}(j+1)T_1/2}(jT_1/2)^{-d/2},\\
    &|p^{a,b}(t-t_1,x,z)| \le 2C_{\CTpabAbsUcq}e^{C_{\CTpabAbsUexp}(j+1)T_1/2}(jT_1/2)^{-d/2},
  \end{align*}
  and, for any $\varepsilon > 0$, there is a constant $R_0 > 0$ such that for all $n \ge 1$ and $y \in \mathbb{R}^d$,
  \begin{equation*}
    \int_{|z-y| \ge R_0} |p^{a,b_n}(t_1,z,y)| dz < \frac{\varepsilon}{8C_{\CTpabAbsUcq}e^{C_{\CTpabAbsUexp}(j+1)T_1/2}(jT_1/2)^{-d/2}}.
  \end{equation*}
  On the other hand, $p^{a,b_n}(t,x,y)$ converges to $p^{a,b}(t,x,y)$ uniformly in $[jT_1/2, (j+1)T_1/2] \times K^{R_0} \times K^{R_0}$ as $n \rightarrow \infty$ since $K^{R_0}$ is bounded. Note that $t - t_1 \in [jT_1/2, (j+1)T_1/2]$. Then for all large enough $n$,
  \begin{equation*}
    \sup_{x,z\in K^{R_0}} |p^{a,b_n}(t-t_1,x,z) - p^{a,b}(t-t_1,x,z)| < \frac{\varepsilon}{2C_{\CTpabAbsUcp}e^{C_{\CTpabAbsUexp}t_1}},
  \end{equation*}
  while by Theorem \ref{thm_pababsup},
  \begin{equation*}
    \int_{\mathbb{R}^d} |p^{a,b_n}(t_1,z,y)| dz \le C_{\CTpabAbsUcp}e^{C_{\CTpabAbsUexp}t_1} \int_{\mathbb{R}^d}p^a({2C_{\CTpaLexp}t_1}/{C_{\CTpaUexp}},z,y) dz = C_{\CTpabAbsUcp}e^{C_{\CTpabAbsUexp}t_1}
  \end{equation*}
   for all $y \in \mathbb{R}^d$. Thus we have for all $(x,y)\in K\times K$
  \begin{align*}
    I_1 \le& \left( \int_{|z-y| \ge R_0}+ \int_{|z-y|< R_0} \right)
    |p^{a,b_n}(t-t_1,x,z) - p^{a,b}(t-t_1,x,z)| |p^{a,b_n}(t_1,z,y)|dz \\
    \le& 4C_{\CTpabAbsUcq}e^{C_{\CTpabAbsUexp}(j+1)T_1/2}(jT_1/2)^{-d/2} \cdot \frac{\varepsilon}{8C_{\CTpabAbsUcq}e^{C_{\CTpabAbsUexp}(j+1)T_1/2}(jT_1/2)^{-d/2}} \\
    & + \int_{K^{R_0}}
    |p^{a,b_n}(t-t_1,x,z) - p^{a,b}(t-t_1,x,z)| |p^{a,b_n}(t_1,z,y)|dz \\
    \le& \frac{\varepsilon}{2} + C_{\CTpabAbsUcp}e^{C_{\CTpabAbsUexp}t_1}\cdot \frac{\varepsilon}{2C_{\CTpabAbsUcp}e^{C_{\CTpabAbsUexp}t_1}}\\
    =& \,\varepsilon.
  \end{align*}
  Similarly, we can get $I_2 < \varepsilon$ for large enough $n$. Thus, we have proved that $p^{a,b_n}(t,x,y)$ converges to $p^{a,b}(t,x,y)$ uniformly in $[(j+1)T_1, (j+2)T_1/2] \times K \times K$ as $n \rightarrow \infty$.
\end{proof}

Lemma \ref{thm_pab'pos} and Lemma \ref{thm_pabncnvg} immediately yield  the following.

\begin{lem}\label{thm_pabpositive}
  Let $M > 0$. For every $a\in (0,M]$,
  \begin{equation*}
    p^{a,b}(t,x,y) \ge 0,\quad t > 0 \text{ and } x,y\in \mathbb{R}^d.
  \end{equation*}
\end{lem}

\medskip

\renewcommand{\proofname}{\bf{Proof of Theorem \ref{T:main}}}
\begin{proof}
  Theorem \ref{T:main} follows from (\ref{equ_pabsmalldef}), Theorem \ref{thm_pabctscsvtv}, Lemma \ref{thm_pabpositive}, Theorem \ref{thm_generator} and Theorem \ref{thm_hkUnique}.
\end{proof}
\renewcommand{\proofname}{Proof}

\section{Lower bound estimates}

In this section, we derive the sharp lower bound of the heat kernel $p^{a,b}(t,x,y)$.
By Lemmas \ref{thm_feller} and \ref{thm_pabpositive},
$P^{a,b}$ is a Feller semigroup in $C_{\infty}(\mathbb{R}^d)$.
Therefore there is
a conservative Feller process $X^{a,b}=\{X_t^{a,b},t\ge 0,\mathbb{P}_x, x\in \mathbb{R}^d\}$
so that
$$
\mathbb{E}_x \left[ f(X_t^{a,b}) \right]
= P_t^{a,b}f(x) = \int_{\R^d} p^{a, b}(t, x, y) f(y) dy
\quad \hbox{for } x\in \mathbb{R}^d \hbox{ and }f\in C_\infty(\mathbb{R}^d).
$$

 The following lemmas will be used to derive the L\'{e}vy system of $X^{a,b}$.

\begin{lem}\label{thm_mar_P}
  For every $f \in \mathbb{K}_{d,1}$,
    $ \lim_{t\rightarrow 0}
    \sup_{x\in \mathbb{R}^d} \int_0^t P^{a,b}_s|f|(x) ds = 0$.
\end{lem}

\begin{proof}
  By Theorem \ref{thm_pababsup}
 and (\ref{equ_qeqvlnt}), for $0<t<1$ and $x\in \mathbb{R}^d$,
  \begin{align*}
    \int_0^t P^{a,b}_s|f|(x)ds \le&  C_{\CTpabAbsUcq} e^{C_{\CTpabAbsUexp}t} \int_0^t \int_{\mathbb{R}^d} q_{d,C_{\CTpaUexp}^2/(2C_{\CTpaLexp})}^a(s,x,y) f(y) dy ds\\
    \le& C_{\CTqEqv}C_{\CTpabAbsUcq} e^{C_{\CTpabAbsUexp}t}\left(\!\!\sqrt{t}\!\! \int_{\mathbb{R}^d} |f(y)|N^{C_{\CTpaUexp}^2/(2C_{\CTpaLexp})}(t,x,y)dy + \!\!\int_0^t\!\! \int_{|x-y|^2\ge s} \!\!\frac{a^\alpha s|f|(y)}{|x-y|^{d+\alpha}} dy ds\right)\\
    \le& C_{\CTqEqv}C_{\CTpabAbsUcq} e^{C_{\CTpabAbsUexp}t} \left(\!\!\sqrt{t}\!\! \sup_{x\in\mathbb{R}^d}\!\!\int_{\mathbb{R}^d} |f(y)|N^{C_{\CTpaUexp}^2/(2C_{\CTpaLexp})}(t,x,y)dy + t^{(3-\alpha)/2} H_{|f|}^{(1+\alpha)/2}(t,x)\right).
  \end{align*}
  Thus, by Lemma \ref{thm_Kequi}
  \begin{equation*}
    \lim_{t\rightarrow 0} \sup_{x\in \mathbb{R}^d} \int_0^t P^{a,b}_s|f|(x) ds \le \lim_{t\rightarrow 0} C_{\CTqEqv}C_{\CTpabAbsUcq} \left(\sqrt{t}N_{|f|}^{C_{\CTpaUexp}^2/(2C_{\CTpaLexp})}(t) + t^{(3-\alpha)/2}\sup_{x\in \mathbb{R}^d} H_{|f|}^{(1+\alpha)/2}(t,x)\right) = 0.
  \end{equation*}
\end{proof}

Using Lemma \ref{thm_feller}, Lemma \ref{thm_mar_P} and Theorem
\ref{thm_generator}, the proof of the following result is very similar to that
of \cite[Theorem 2.5]{ChenKimSong.2012} so it is omitted.

\begin{lem}\label{thm_martingale1}
  Suppose $M > 0$. For every $a\in (0,M]$, $x\in \mathbb{R}^d$ and every $f\in C_c^\infty(\mathbb{R}^d)$,
  \begin{equation*}
    M_t^f:= f(X_t^{a,b}) - f(X_0^{a,b}) - \int_0^t \mathcal{L}^{a,b}f(X_s^{a,b}) ds
  \end{equation*}
  is a martingale under $\mathbb{P}_x$.
\end{lem}

Using above lemma, the proof of next theorem is very similar to that
for \cite[Lemma 4.7]{ChenKumagai.2003} and \cite[Appendix A]{ChenKumagai.2008};
see \cite[Theorem 2.6]{ChenKimSong.2012} for some details.

\begin{thm}\label{thm_levysystem}
  For $M > 0$ and every $a\in (0,M]$, $X^{a,b}$ has the same L\'{e}vy system as $Z^a$, that is for any $x \in \mathbb{R}^d$, any non-negative measure function $f$ on $\mathbb{R}_+\times \mathbb{R}^d\times \mathbb{R}^d$ vanishing on $\{(s,x,y) \in \mathbb{R}_+\times \mathbb{R}^d\times \mathbb{R}^d: x=y\}$ and stopping time $S$ (with respect to the filtration of $X^{a,b}$),
  \begin{equation*}
    \mathbb{E}_x\left[\sum_{s\le S} f(s,X_{s-}^{a,b},X_{s}^{a,b})\right] = \mathbb{E}_x \left[\int_0^S \int_{\mathbb{R}^d} f(s,X_{s}^{a,b},y) J^a(X_s^{a,b},y)dy ds\right],
  \end{equation*}
\end{thm}

For an open set  $U\subset \mathbb{R}^d$, define
$$
\tau^{a,b}_U := \inf\{t>0:X^{a,b}_t \notin U\} \quad \hbox{and} \quad
 \sigma_{U}^{a,b} = \inf\{t \ge 0: X_t^{a,b} \in U\}.
$$

\begin{lem}\label{thm_exittime}
  For each $M > 0$ and $R_0 > 0$, there is a constant $\kappa = \kappa(d,\alpha,M,R_0,b) < 1$ depending on $b$ only via the rate at which $M_b(r)$ goes to zero such that for all $a \in (0,M]$, $r \in (0,R_0]$ and all $x\in \mathbb{R}^d$,
  \begin{equation}\label{e:5.3}
    \mathbb{P}_x\left(\tau_{B(x,r)}^{a,b} \le \kappa r^2 \right) \le \frac{1}{2}.
  \end{equation}
\end{lem}
\begin{proof}
  By the strong Markov property of $X^{a,b}$ (See \cite[Exercise (8.17), pp. 43-44]{BlumenthalGetoor.1968}), for $x\in \mathbb{R}^d$ and $t>0$, we have
  \begin{align*}
    \mathbb{P}_x\left(\tau_{B(x,r)}^{a,b} \le t\right) =& \mathbb{P}_x\left(\tau_{B(x,r)}^{a,b} \le t, X_t^{a,b} \in B(x,r/2)\right) + \mathbb{P}_x\left(\tau_{B(x,r)}^{a,b} \le t, X_t^{a,b} \in B(x,r/2)^c\right)\\
    \le&\mathbb{E}_x\Big[\mathbb{P}_{X_{\tau_{B(x,r)}^{a,b}}^{a,b}}\left(\left|X_{t-\tau_{B(x,r)}^{a,b}}^{a,b} - X_0^{a,b} \right| \ge r/2\right), \tau_{B(x,r)}^{a,b} \le t\Big]\\
    &\quad + \mathbb{P}_x\left(\left|X_{t}^{a,b} - X_0^{a,b} \right| \ge r/2\right)\\
    \le& 2\sup_{s\le t}\sup_{x\in \mathbb{R}^d} \mathbb{P}_x\left(\left|X_{s}^{a,b} - X_0^{a,b} \right| \ge r/2\right).
  \end{align*}
   By \eqref{equ_pabAbsUp}
  with $T=R_0^2$, for $t \in (0,R_0^2]$, there are positive constants $c_i,i=1,2,3$ depending only on $d,\alpha,M,R_0$ such that
  \begin{eqnarray*}
    &&\sup_{s\le t}\sup_{x\in \mathbb{R}^d} \mathbb{P}_x\left(\left|X_{s}^{a,b} - X_0^{a,b} \right| \ge r/2\right) \\
    &\stackrel{c_1e^{c_2R_0^2}}{\lesssim}&  \sup_{s\le t}\sup_{x\in \mathbb{R}^d} \int_{|x-y|\ge r/2} \left(s^{-d/2}\exp\left(-\frac{c_3|x-y|^2}{s}\right) + s^{-d/2}\wedge \frac{a^\alpha s}{|x-y|^{d+\alpha}} \right) dy\\
    &\stackrel{\omega_d}{\lesssim}& \sup_{s\le t}\int_{\frac{r}{2\sqrt{s}}}^\infty \left( e^{-c_3\rho^2} + 1 \wedge \frac{M^\alpha s^{1-\alpha/2}}{\rho^{d+\alpha}} \right)\rho^{d-1} d\rho \\
    &\le& \int_{\frac{r}{2\sqrt{t}}}^\infty \left( e^{-c_3\rho^2} + 1 \wedge \frac{M^\alpha R_0^{2-\alpha}}{\rho^{d+\alpha}} \right)\rho^{d-1} d\rho.
  \end{eqnarray*}
  Setting $t =\kappa r^2$ in the last display, where $\kappa \in (0,1)$ is undetermined, we have
  \begin{align*}
    \mathbb{P}_x\left(\tau_{B(x,r)}^{a,b} \le \kappa r^2\right) \le& 2c_1e^{c_2R_0^2}\omega_d \int_{\frac{1}{2\sqrt{\kappa}}}^\infty \left( e^{-c_3\rho^2} + 1 \wedge \frac{M^\alpha R_0^{2-\alpha}}{\rho^{d+\alpha}} \right)\rho^{d-1} d\rho,
  \end{align*}
  which goes to $0$ as $\kappa \rightarrow 0$. Thus we can choose $\kappa < 1$
  so that \eqref{e:5.3} holds.
\end{proof}

 \begin{lem}\label{thm_entrytime}
  For each $M > 0$ and $R_0 > 0$, there is a constant $c_1= c_1(d,\alpha,M,R_0,b)$ depending on $b$ only via the rate at which $M_b(r)$ goes to zero, such that for all $r \in (0,R_0]$ and $x,y \in \mathbb{R}^d$ with $|x-y| \ge 2r$,
  \begin{equation*}
    \mathbb{P}_x\left(\sigma_{B(y,r)}^{a,b} < \kappa r^2 \right) \ge c_1r^{d+2} \frac{a^\alpha}{|x-y|^{d+\alpha}}.
  \end{equation*}
\end{lem}
\begin{proof}
  By Lemma \ref{thm_exittime},
  \begin{equation*}
    \mathbb{E}_x\left[\frac{\kappa r^2}{2}\wedge \tau_{B(x,r)}^{a,b}\right] \ge \frac{\kappa r^2}{2} \mathbb{P}_x\left(\tau_{B(x,r)}^{a,b} \ge \frac{\kappa r^2}{2}\right) \ge \frac{\kappa r^2}{4}.
  \end{equation*}
  By Theorem \ref{thm_levysystem}, we have
  \begin{align*}
    \mathbb{P}_x\left(\sigma_{B(y,r)}^{a,b} < \kappa r^2 \right) &\ge \mathbb{P}_x\left(X_{\frac{\kappa r^2}{2} \wedge \tau_{B(x,r)}^{a,b}}^{a,b} \in B(y,r)\right)\\
    &= \mathbb{E}_x\left(\int_0^{\frac{\kappa r^2}{2} \wedge \tau_{B(x,r)}^{a,b}}\int_{B(y,r)} J^{a}(X_s^{a,b},u) du ds\right)\\
    &\ge 2^{-(d+\alpha)}\mathbb{E}_x\left[\frac{\kappa r^2}{2} \wedge \tau_{B(x,r)}^{a,b}\right] \int_{B(y,r)} \frac{a^\alpha}{|x-y|^{d+\alpha}} du\\
    &\ge \frac{\Omega_d}{4\cdot2^{d+\alpha}} \kappa r^{d+2} \frac{a^\alpha}{|x-y|^{d+\alpha}},
  \end{align*}
  where $\Omega_d$ is the volume of unit ball in $\mathbb{R}^d$, and in the second to the last inequality, we have used the fact that for $u \in B(y,r)$, $|u - X_s^{a,b}| \le 2r+|x-y| \le 2|x-y|$.
\end{proof}

\const{\CTpabSmallLc}
\begin{lem}\label{thm_pablowpoly}
  For every $M > 0$, there is a constant $C_{\CTpabSmallLc} = C_{\CTpabSmallLc}(d,\alpha,M,b)$ depending on $b$ only via the rate at which $M_b(r)$ goes to zero, such that for all $t\in (0,t_*], a\in (0,M]$ and $x,y\in \mathbb{R}^d$
  \begin{equation*}
    p^{a,b}(t,x,y) \ge C_{\CTpabSmallLc} \left(t^{-d/2}\wedge\frac{a^\alpha t}{|x-y|^{d+\alpha}}\right).
  \end{equation*}
\end{lem}
\begin{proof}
  By (\ref{equ_pablow}), for $t \in (0,t_*]$ and $x,y\in \mathbb{R}^d$, with $|x-y|^2 \le t$
  \begin{equation*}
    p^{a,b}(t,x,y) \ge C_{\CTpnSum}^{-1}t^{-d/2}\ge C_{\CTpnSum}^{-1}\left(t^{-d/2}\wedge\frac{a^\alpha t}{|x-y|^{d+\alpha}}\right).
  \end{equation*}
  It remains to consider the case $|x-y|^2 > t$. For any $t \in (0,t_*]$, by the strong Markov property, Lemma \ref{thm_exittime} and Lemma \ref{thm_entrytime} with $R_0 = \sqrt{t_*}$ and $r= \sqrt{t}/4$, we have $|x-y| > \sqrt{t} > 2r$ and
  \begin{align*}
    &~\mathbb{P}_x\left(X_{\kappa t/16}^{a,b} \in B(y, \sqrt{t}/2)\right)\\
    \ge&~ \mathbb{P}_x\left(X^{a,b} \text{ hits } B(y,\sqrt{t}/2) \text{ before time } \kappa t/16 \text{ and stays there for at least } \kappa t/16 \text{ units of time}\right)\\
    \ge&~ \mathbb{P}_x\left(\sigma^{a,b}_{B(y,\sqrt{t}/4)} \le \kappa t/16,\tau^{a,b}_{B(y,\sqrt{t}/2)}\circ \theta_{\sigma^{a,b}_{B(y,\sqrt{t}/4)}} \ge \kappa t/16\right) \\
    \ge&~ \mathbb{P}_x\left(\sigma_{B(y,\sqrt{t}/4)}^{a,b} < \kappa t/16\right) \inf_{z \in B(y,\sqrt{t}/4)}
    \mathbb{P}_z\left(\tau_{B(y,\sqrt{t}/2)}^{a,b} \ge \kappa t/16\right)\\
    \ge&~ \mathbb{P}_x\left(\sigma_{B(y,\sqrt{t}/4)}^{a,b} < \kappa t/16\right) \inf_{z \in B(y,\sqrt{t}/4)} \mathbb{P}_z\left(\tau_{B(z,\sqrt{t}/4)}^{a,b} \ge \kappa t/16\right)\\
    \ge&~ c_1t^{(d+2)/2}\frac{a^\alpha}{|x-y|^{d+\alpha}},
  \end{align*}
  for some constant $c_1 = c_1(d,\alpha,M,b) > 0$. Combining this with Lemma \ref{thm_exittime} and Chapman-Kolmogorov equation (\ref{e:1.8}), we have
  for $t\in (0,t_*]$
  \begin{align*}
    p^{a,b}(t,x,y) =& \int_{\mathbb{R}^d} p^{a,b}(\kappa t/16,x,z)p^{a,b}((1-\kappa/16)t,z,y) dz\\
    \ge& \int_{B(y,\sqrt{t}/2)} p^{a,b}(\kappa t/16,x,z)p^{a,b}((1-\kappa/16)t,z,y) dz\\
    \ge& \inf_{z \in B(y,\sqrt{t}/2)} p^{a,b}((1-\kappa/16)t,z,y)\mathbb{P}_x\left(X_{\kappa t/16}^{a,b}\in B(y,\sqrt{t}/2)\right)\\
    \ge& c_2t^{-d/2} t^{(d+2)/2} \frac{a^\alpha}{|x-y|^{d+\alpha}}\\
    =& c_2 \frac{a^\alpha t}{|x-y|^{d+\alpha}}\ge c_2\left(t^{-d/2}\wedge\frac{a^\alpha t}{|x-y|^{d+\alpha}}\right),
  \end{align*}
  where $c_2 = c_2(d,\alpha,M,b)$ is a positive constant and in the third to the last  inequality, we have used the fact that $\kappa < 1$ and for $z\in B(y,\sqrt{t}/2)$, $|z-y|^2 < t/ 4 < (1-{\kappa/ 16}) t$.
\end{proof}

\const{\CTpabSmallLexpI}
\const{\CTpabSmallLexpII}
\begin{lem}\label{thm_pablowexp}
  Suppose $M > 0$. For all $a \in (0,M]$, $t \in (0,t_*]$ and $x,y\in \mathbb{R}^d$, there are constants $C_i=C_i(d,\alpha,M) > 0,i=\CTpabSmallLexpI,\CTpabSmallLexpII$ such that
  \begin{equation*}
    p^{a,b}(t,x,y) \ge C_{\CTpabSmallLexpI}t^{-d/2} \exp\left(-\frac{C_{\CTpabSmallLexpII}|x-y|^2}{t}\right).
  \end{equation*}
\end{lem}
\begin{proof}
  By (\ref{equ_pablow}), for all $t \in (0,t_*]$ and $x,y\in \mathbb{R}^d$ with $|x-y|^2 < t$, we have
  \begin{equation}\label{equ_pablowcase1}
    p^{a,b}(t,x,y) \ge C_{\CTpnSum}^{-1}t^{-d/2}.
  \end{equation}
  Next, we consider the case $|x-y|^2 > t$. We fix $x,y \in \mathbb{R}^d$ with $|x-y|^2 \ge t$. Let $k$ be the smallest integer such that $9|x-y|^2/t < k$. Set $\xi_j = x+\frac{j-1}{k}(y-x), 1\le j \le k-1$ and $A = \prod_{j=1}^{k-1}B(\xi_j, \frac{\sqrt{t}}{3\sqrt{k}})$. For any $(x_1,\cdots,x_{k-1}) \in A$, we have
  $|x - x_1| < \frac{\sqrt{t}}{3\sqrt{k}} < \frac{\sqrt{t}}{\sqrt{k}}$,
$$ \max_{1<j\le k-1}|x_j - x_{j-1}| = \max_{1<j\le k-1}\left| x_j-\xi_j + \xi_{j-1}-x_{j-1} + \frac{y-x}{k}\right|  < \frac{\sqrt{t}}{3\sqrt{k}} + \frac{\sqrt{t}}{3\sqrt{k}} + \frac{\sqrt{t}}{3\sqrt{k}} = \frac{\sqrt{t}}{\sqrt{k}}
$$
and $|x_{k-1}-y| = |x_{k-1} - \xi_{k-1} + \xi_{k-1} - y| < \frac{\sqrt{t}}{\sqrt{k}}$.
  Hence by Lemma \ref{thm_pabpositive}, Chapman-Kolmogorov equation (\ref{e:1.8}) and (\ref{equ_pablowcase1}),
  \begin{align*}
    p^{a,b}(t,x,y) &= \int_{\mathbb{R}^{d(k-1)}} p^{a,b}(\frac{t}{k},x,x_1) \cdots p^{a,b}(\frac{t}{k},x_{k-1},y)
    dx_1dx_2\cdots dx_{k-1}\\
    &\ge \int_{A} p^{a,b}(\frac{t}{k},x,x_1) \cdots p^{a,b}(\frac{t}{k},x_{k-1},y) dx_1dx_2\cdots dx_{k-1}\\
    &\ge C_{\CTpnSum}^{-k}\left(\frac{t}{k}\right)^{-dk/2} \Omega_d^{k-1}\left(\frac{\sqrt{t}}{3\sqrt{k}}\right)^{d(k-1)}\\
    &= t^{-d/2}\frac{k^{d/2}}{C_{\CTpnSum}} \left(\frac{k^{d/2}}{C_{\CTpnSum}}\frac{\Omega_d}{3^dk^{d/2}}\right)^{k-1}\\
    &= \frac{k^{d/2}}{C_{\CTpnSum}} t^{-d/2} \left(\frac{\Omega_d}{C_{\CTpnSum}3^d}\right)^{k-1}\\
    &\ge \frac{3^d}{C_{\CTpnSum}} t^{-d/2} \exp\left(-\ln\frac{C_{\CTpnSum}3^d}{\Omega_d} \frac{9|x-y|^2}{t}\right),
  \end{align*}
  where $\Omega_d$ is the volume of unit ball in $\mathbb{R}^d$. This together with (\ref{equ_pablowcase1}) proves the lemma  with $C_{\CTpabSmallLexpI} := \frac{3^d}{C_{\CTpnSum}}$ and $C_{\CTpabSmallLexpII} := 9\ln\frac{C_{\CTpnSum}3^d}{\Omega_d}$.
\end{proof}

\renewcommand{\proofname}{\bf{Proof of Theorem \ref{T:1.3}}}
\begin{proof}
  The upper bound of $p^{a,b}(t,x,y)$ is shown
     by Theorem \ref{thm_pababsup} and Lemma \ref{thm_pabpositive}.
   We need only to show the lower bound. Without loss of generality, we assume $T > t_*$. If $t \in (0,t_*]$, by Lemma \ref{thm_pablowpoly} and Lemma \ref{thm_pablowexp}, there is a constant $c_1=c_1(d,\alpha,M,b) > 0$ such that for $x,y \in \mathbb{R}^d$
  \begin{equation}\label{equ_pablower1}
    \begin{split}
      p^{a,b}(t,x,y) \ge& \frac{1}{2} \left(C_{\CTpabSmallLexpI}t^{-d/2}\exp\left(-\frac{C_{\CTpabSmallLexpII}|x-y|^2}{t}\right) + C_{\CTpabSmallLc}\left(t^{-d/2}\wedge \frac{a^\alpha t}{|x-y|^{d+\alpha}}\right)\right)\\
      \ge&c_1 \left(t^{-d/2}\exp\left(-\frac{C_{\CTpabSmallLexpII}|x-y|^2}{t}\right) +t^{-d/2}\wedge \frac{a^\alpha t}{|x-y|^{d+\alpha}}\right)\\
      =& c_1q_{d,C_{\CTpabSmallLexpII}}^a(t,x,y).
    \end{split}
  \end{equation}
  If $t \ge t_*$, we let $k$ be the smallest integer such that $t_*k \ge t >(k-1)t_*$. Note that by Theorem \ref{thm_pabound}, for $t\in (0,T]$ and $x,y\in \mathbb{R}^d$,
  \begin{equation*}
    q_{d,C_{\CTpabSmallLexpII}}^a(t,x,y) \ge \left(\frac{C_{\CTpaUexp}}{C_{\CTpabSmallLexpII}}\right)^{d/2} q_{d,C_{\CTpaUexp}}^a(\frac{C_{\CTpaUexp}t}{C_{\CTpabSmallLexpII}},x,y) \ge \left(\frac{C_{\CTpaUexp}}{C_{\CTpabSmallLexpII}}\right)^{d/2}C_{\CTpaUc}^{-1}p^a(\frac{C_{\CTpaUexp}t}{C_{\CTpabSmallLexpII}},x,y).
  \end{equation*}
  Using this,  (\ref{e:1.8}), (\ref{equ_pablower1}) and Theorem \ref{thm_pabound}, we have
  \begin{align*}
    p^{a,b}(t,x,y) &\ge c_1^{-k} \int_{\mathbb{R}^{d(k-1)}} q_{d,C_{\CTpabSmallLexpII}}^a(\frac{t}{k},x,x_1) \cdots q_{d,C_{\CTpabSmallLexpII}}^a(\frac{t}{k},x_{k-1},y) dx_1\cdots dx_{k-1}\\
    &\ge c_1^{-k} \left(\frac{C_{\CTpaUexp}}{C_{\CTpabSmallLexpII}}\right)^{dk/2}C_{\CTpaUc}^{-k}\int_{\mathbb{R}^{d(k-1)}} p^a(\frac{C_{\CTpaUexp}}{C_{\CTpabSmallLexpII}}\frac{t}{k},x,x_1) \cdots p^a(\frac{C_{\CTpaUexp}}{C_{\CTpabSmallLexpII}}\frac{t}{k},x_{k-1},y) dx_1\cdots dx_{k-1}\\
    &= c_1^{-k}\left(\frac{C_{\CTpaUexp}}{C_{\CTpabSmallLexpII}}\right)^{dk/2}C_{\CTpaUc}^{-k} p^a(\frac{C_{\CTpaUexp}t}{C_{\CTpabSmallLexpII}},x,y)\\
    &\ge \frac{C_{\CTpaUexp}C_{\CTpaLc}}{c_1C_{7}}\left(\frac{C_{\CTpaUexp}}{c_1C_{7}}\right)^{d(k-1)/2} q_{d,C_{\CTpaLexp}}^a(\frac{C_{\CTpaUexp}t}{C_{\CTpabSmallLexpII}},x,y)\\
    &\ge \frac{C_{\CTpaUexp}C_{\CTpaLc}c_2}{c_1C_{7}}\left(\frac{C_{\CTpaUexp}}{c_1C_{7}}\right)^{dt/(2t_*)} q_{d,C_{\CTpaLexp}C_{\CTpabSmallLexpII}/C_{\CTpaUexp}}^a(t,x,y)\\
    &\ge \frac{C_{\CTpaUexp}C_{\CTpaLc}c_2}{c_1C_{7}}\left(\frac{C_{\CTpaUexp}}{c_1C_{7}}\right)^{dT/(2t_*)} q_{d,C_{\CTpaLexp}C_{\CTpabSmallLexpII}/C_{\CTpaUexp}}^a(t,x,y).
  \end{align*}
  where $c_2=c_2(d,\alpha,M,b)$ is a positive constant. This completes the proof.
\end{proof}
\renewcommand{\proofname}{Proof}

\section{Martingale problem and L\'evy process with drift}

Following the approach in \cite{ChenWang.2013a}, we can show that the martingale problem for $(\LL^{a,b}, C^\infty_c(\R^d))$ is well-posed, and there is a unique weak solution to SDE \eqref{e:1.1}.

\const{\CTulambda}
For $a > 0$ and $\lambda > 0$, define
\begin{equation*}
  u^a_\lambda(x) = \int_0^\infty e^{-\lambda t} p^a(t,x)dt,\quad x\in \mathbb{R}^d.
\end{equation*}

\begin{lem}\label{thm_uaUp}
  There is a constants $C_{\CTulambda} = C_{\CTulambda}(d)$ such that for all $a > 0$, $\lambda \ge 1$ and $x\in \mathbb{R}^d$, we have
  \begin{equation}\label{eq_uaUp}
    u^a_\lambda(x) \le C_{\CTulambda} (1\vee a^\alpha)\left\{
      \begin{aligned}
        &\frac{1}{|x|^{d-1}}\wedge \frac{\lambda^{-\frac{\alpha+1}{2}}}{|x|^{d+\alpha}},\quad d = 2,\\
        &\frac{1}{|x|^{d-2}}\wedge \frac{\lambda^{-\frac{\alpha+2}{2}}}{|x|^{d+\alpha}},\quad d > 2,
      \end{aligned}\right.
  \end{equation}
  and
  \begin{equation}\label{eq_graduaUp}
    \left|\nabla u^a_\lambda(x)\right| \le C_{\CTulambda} (1\vee a^\alpha) \left(\frac{1}{|x|^{d-1}}\wedge \frac{\lambda^{-\frac{\alpha+2}{2}}}{|x|^{d+1+\alpha}}\right).
  \end{equation}
\end{lem}

\begin{proof}
  Note that for each $\theta > 0$, the function $\psi(t) = t^\theta e^{-t}$ on $[0, \infty)$  is bounded by $\theta^\theta e^{-\theta}$. By (\ref{equ_paSharpBnds}), we have
  \begin{equation}\label{eq_uaUp1}
    \begin{split}
      u^a_\lambda(x) &\le C_{\CTpaBndI} \int_0^\infty e^{-\lambda t} \left(t^{-d/2}e^{-C_{\CTpaBndII}|x|^2/t}+(a^\alpha t)^{-d/\alpha}\wedge \frac{a^\alpha t}{|x|^{d+\alpha}}\right) dt\\
      &\le c_1  \int_0^\infty e^{-\lambda t}\left(\frac{t}{|x|^{d+2}} + \frac{a^\alpha t}{|x|^{d+\alpha}}\right) dt\\
      &= c_1 \lambda^{-2}\left(\frac{1}{|x|^{d+2}} + \frac{a^\alpha}{|x|^{d+\alpha}}\right)\\
      &\le c_1 (1\vee a^\alpha)\lambda^{-2}\left(\frac{1}{|x|^{d+2}} + \frac{1}{|x|^{d+\alpha}}\right).
    \end{split}
  \end{equation}
  Since $\lambda \ge 1$, if $|x|^2 \ge  1/\lambda$,
  \begin{equation}\label{eq_uaUp2}
    u^a_\lambda(x) \le 2c_1 (1\vee a^\alpha)\frac{\lambda^{-\frac{\alpha+2}{2}}}{|x|^{d+\alpha}}.
  \end{equation}
  When $|x|^2 < 1/\lambda$, similar to (\ref{eq_uaUp1}), we have
  \begin{equation}\label{eq_uaUp3}
      \int_0^{|x|^2} e^{-\lambda t} p^a(t,x)dt \le c_1 \int_0^{|x|^2} \left(\frac{t}{|x|^{d+2}} + \frac{a^\alpha t}{|x|^{d+\alpha}}\right) dt \le \frac{c_1}{2} \left(\frac{1}{|x|^{d-2}} + \frac{a^\alpha}{|x|^{d-4+\alpha}}\right) \le  \frac{c_1 (1\vee a^\alpha)}{|x|^{d-2}}
  \end{equation}
  and
  \begin{equation}\label{eq_uaUp4}
    \begin{split}
      \int_{|x|^2}^\infty e^{-\lambda t} p^a(t,x)dt &\le C_{\CTpaBndI} \int_{|x|^2}^\infty e^{-\lambda t}t^{-d/2} dt\\
      &\le C_{\CTpaBndI} \left\{
        \begin{aligned}
          \frac{1}{|x|} \int_{|x|^2}^\infty e^{-t}t^{-1/2} dt \le \frac{\sqrt{\pi}}{|x|} \quad \text{ if } d = 2,\\
          \int_{|x|^2}^\infty t^{-d/2} dt = \frac{2}{d-2}\frac{1}{|x|^{d-2}}
           \quad \text{ if } d > 2 .
        \end{aligned}\right.
    \end{split}
  \end{equation}
  Therefore, (\ref{eq_uaUp}) follows from (\ref{eq_uaUp2})-(\ref{eq_uaUp4}). Finally, (\ref{eq_graduaUp}) follows from (\ref{equ_ghk}) and (\ref{eq_uaUp}).
\end{proof}

For $a > 0$ and $\lambda > 0$, define the resolvent operator $U^a_\lambda$ by
\begin{equation*}
  U^a_\lambda g(x) = \int_{\mathbb{R}^d} u^a_\lambda(x-y) g(y) dy = \int_{\mathbb{R}^d} u^a_\lambda(y) g(x-y) dy, \quad g\in C_b(\mathbb{R}^d), x\in \mathbb{R}^d.
\end{equation*}

\medskip

 Let $C_\infty^\infty (\mathbb{R}^d)$ be the collection of the smooth functions on $\mathbb{R}^d$ that together with their partial derivatives of any order vanish at infinity.

\begin{lem}\label{thm_UaBound}
  For every $a > 0$ and $\lambda \ge 1$, $U^a_\lambda$ and $\nabla U^a_\lambda$ are bounded operators on $C_\infty(\mathbb{R}^d)$. Moreover, $U^a_\lambda f \in C_\infty^\infty(\mathbb{R}^d)$ for every $f\in C_\infty^\infty(\mathbb{R}^d)$.
\end{lem}

\begin{proof}
  By (\ref{eq_graduaUp}), we have for every $a > 0$, $\lambda\ge 1$, $f\in C_\infty(\mathbb{R}^d)$ and $x \in \mathbb{R}^d$,
  \begin{equation*}
    \int_{\mathbb{R}^d} \left|\nabla u^a_\lambda(y)\right| |f(x-y)|dy \le C_{\CTulambda}(1\vee a^\alpha)\|f\|_\infty\int_{\mathbb{R}^d} \frac{1}{|y|^{d-1}}\wedge \frac{\lambda^{-\frac{\alpha+2}{2}}}{|y|^{d+1+\alpha}}dy < \infty.
  \end{equation*}
  Combining this with the fact that $u^a_\lambda$ in continuously differentiable off the origin and the dominated convergence theorem, we have
  \begin{equation*}
    \nabla U^a_\lambda f(x) = \int_{\mathbb{R}^d} \nabla u^a_\lambda(x-y) f(y) dy = \int_{\mathbb{R}^d} \nabla u^a_\lambda(y) f(x-y) dy.
  \end{equation*}
  Since both $u^a_\lambda$ and $\nabla u^a_\lambda$ are integrable over $\mathbb{R}^d$ and $f(x-y)$ converges to $0$ as $|x|\rightarrow \infty$, we have that both $U^a_\lambda f$ and $\nabla U^a_\lambda f$ are in $C_\infty(\mathbb{R}^d)$ and
  \begin{equation*}
    \left\|U^a_\lambda f\right\|_\infty \le C_{\CTulambda}(1\vee a^\alpha)\|f\|_\infty, \text{ and } \left\|\nabla U^a_\lambda f\right\|_\infty \le C_{\CTulambda}(1\vee a^\alpha)\|f\|_\infty,
  \end{equation*}
  where $C_{\CTulambda}$ is the constant from Lemma \ref{thm_uaUp}.
  Similarly, by the dominated convergence theorem, for $f \in C_\infty^\infty(\mathbb{R}^d)$, we have
  \begin{equation*}
    \partial^{k_1}_{x_1}\cdots \partial^{k_d}_{x_d} U^a_\lambda f(x) = \int_{\mathbb{R}^d} u^a_\lambda(y) \partial^{k_1}_{x_1}\cdots \partial^{k_d}_{x_d} f(x-y) dy,
  \end{equation*}
  which shows that $U^a_\lambda f\in C_\infty^\infty(\mathbb{R}^d)$.
\end{proof}

\begin{lem}\label{thm_nablaUa}
  Suppose that $M > 0$ and $b \in \mathbb{K}_{d,1}$. There is a constant $\lambda_0 = \lambda_0(d,\alpha,M,b)\ge 1$ with the dependence on $b$ only via the rate at which $M_b(r)$ goes to zero such that for every $a \in (0,M]$, $\lambda \ge \lambda_0$ and $f \in C_\infty(\mathbb{R}^d)$,
  \begin{equation*}
    \left\|\nabla U^a_\lambda(bf)\right\|_\infty \le \frac{1}{2} \|f\|_\infty.
  \end{equation*}
\end{lem}
\begin{proof}
  By (\ref{equ_HMb})(with $\beta = \frac{\alpha+2}{2}$) and (\ref{eq_graduaUp}), we have for $a \in (0,M]$, $\lambda \ge \lambda_0$ and $f \in C_\infty(\mathbb{R}^d)$,
  \begin{align*}
    \left\|\nabla U^a_\lambda(bf)\right\|_\infty &\le C_{\CTulambda} (1\vee a^\alpha) \sup_{x\in \mathbb{R}^d} \int_{\mathbb{R}^d}\bigg(\frac{1}{|x-y|^{d-1}}\wedge \frac{\lambda^{-\frac{\alpha+2}{2}}}{|x-y|^{d+1+\alpha}}\bigg) |b(y)||f(y)|dy\\
    &\le C_{\CTulambda}c_1\|f\|_\infty(1\vee M^\alpha) M_b(\lambda^{-1/2}).
  \end{align*}
  Since $b \in \mathbb{K}_{d,1}$, we can choose $\lambda_0\ge 1$ such that $C_{\CTulambda}c_1(1\vee M^\alpha) M_b(\lambda^{-1/2}) \le 1/2$ for every $\lambda > \lambda_0$. This completes the proof.
\end{proof}

\const{\CTbigUlambda}
By \eqref{equ_pabAbsUp} for $\lambda > C_{\CTpabAbsUexp}C_{\CTpaUexp}/(2C_{\CTpaLexp})$,
\begin{equation*}
  \begin{split}
    \mathbb{E}_x\left[\int_0^\infty e^{-\lambda t}|b(X_t)|dt\right] &\le C_{\CTpabAbsUcp}C_{\CTpaUexp}/(2C_{\CTpaLexp}) \int_{\mathbb{R}^d}\int_0^\infty e^{-(\lambda-C_{\CTpabAbsUexp}C_{\CTpaUexp}/(2C_{\CTpaLexp}))t}p^a(t,x,y)dt |b(y)|dy\\
    &= C_{\CTpabAbsUcp}C_{\CTpaUexp}/(2C_{\CTpaLexp})\int_{\mathbb{R}^d}u^a_{\lambda-C_{\CTpabAbsUexp}C_{\CTpaUexp}/(2C_{\CTpaLexp})}(x-y) |b(y)| dy.
  \end{split}
\end{equation*}
Similar to Lemma \ref{thm_nablaUa}, by (\ref{eq_uaUp}), there is a constant $C_{\CTbigUlambda} > C_{\CTpabAbsUexp}C_{\CTpaUexp}/(2C_{\CTpaLexp})\vee 1$ so
 that for every $a\in (0,M]$ and $\lambda > C_{\CTbigUlambda}$,
\begin{equation}\label{eq_UabFinite}
  \sup_{x\in \mathbb{R}^d} U^a_\lambda|b|(x) = \sup_{x\in \mathbb{R}^d}\mathbb{E}_x\left[\int_0^\infty e^{-\lambda t}|b(X_t)|dt\right] < \infty.
\end{equation}
By increasing the value of $\lambda_0$ in Lemma \ref{thm_nablaUa} if needed, we may and do assume that $\lambda_0 \ge C_{\CTbigUlambda}$.

\begin{thm}[Uniqueness]\label{thm_marUni}
  For each $x\in \mathbb{R}^d$ and $a\in (0,M]$, $\mathbb{P}_x$ is the unique solution to the martingale problem for $(\mathcal{L}^{a,b},C_c^\infty(\mathbb{R}^d))$ with initial value $x$.
\end{thm}
\begin{proof}
  Recall that $X_t$ be the coordinate map on $\mathbb{D}([0,\infty),\mathbb{R}^d)$. Using
  Lemmas \ref{thm_uaUp}-\ref{thm_nablaUa}
  and \eqref{eq_UabFinite}, we can finish the proof by repeating the arguments in the proof of \cite[Theorem 2.3]{ChenWang.2013a} except using the following It\^{o}'s formula in place of that in  Step (ii) of \cite[Theorem 2.3]{ChenWang.2013a}:
  \begin{eqnarray*}
      e^{-\lambda t}f(X_t) &=& f(X_0) + \int_0^t e^{-\lambda s} dM_s^f  + \int_0^t e^{-\lambda s}\left(\Delta f(X_s)+ a^\alpha \Delta^{\alpha/2} f(X_s) + b(X_s)\cdot \nabla f(X_s)\right) ds \\
      && - \lambda \int_0^t e^{-\lambda s}f(X_s) ds.
  \end{eqnarray*}
\end{proof}

\renewcommand{\proofname}{\bf{Proof of Theorem \ref{T:1.4}}}
\begin{proof}
  Theorem \ref{thm_marUni} implies that the martingale problem for $(\mathcal{L}^{a,b}, C^\infty_c(\mathbb{R}^d))$ is well-posed. The rest follows from Theorem \ref{T:main}.
\end{proof}

\renewcommand{\proofname}{Proof}

The following theorem establishes the existence of the weak solution of SDE \eqref{e:1.1}.

\begin{thm}[Existence]\label{thm_solutionExist}
  For every $a > 0$, there is a process $Z^a$ defined on $\Omega$ so that all its paths are right continuous and admit left limits, and
  \begin{equation*}
    X^{a,b}_t = x + Z^a_t + \int_0^t b(X^{a,b}_s) ds, \quad t\ge 0.
  \end{equation*}
\end{thm}
\begin{proof}
  The proof is almost the same to that of \cite[Theorem 3.1]{ChenWang.2013a}, except that we use the following arguments instead of those at the beginning of Page 13 in \cite{ChenWang.2013a}: for any $f \in C^\infty_c(\mathbb{R}^d)$,
  \begin{eqnarray*}
  &&\int_0^t b(X^{a,b}_s)\nabla f(X^{a,b}_s) ds + \int_0^t\Delta f(X^{a,b}_s)ds  \\
   &=& \int_0^t \nabla f(X^{a,b}_s) dA_s
    + \frac{1}{2}\sum_{i,j=1}^d \int_0^t\frac{\partial^2f}{\partial x_i\partial x_j}(X^{a,b}_s)d \<M^i,M^j \>_s,
  \end{eqnarray*}
  which implies that
  \begin{equation*}
    A_t = \int_t b(X^{a,b}_t) ds \quad \text{ and } \quad \<M^i,M^j\>_t = \delta_{ij} t .
   \end{equation*}
   Here $\delta_{ij}=1$ if $i=j$ and $\delta_{ij}=0$ if $i\not=j$.
\end{proof}

\renewcommand{\proofname}{\bf{Proof of Theorem \ref{T:1.5}}}
\begin{proof}
  The existence of weak solution to SDE \eqref{e:1.1} follows from Lemma \ref{thm_solutionExist}. Every weak solution to \eqref{e:1.1} solves the martingale problem for $(\mathcal{L}^{a,b},C^\infty_c(\mathbb{R}^d))$ by It\^{o}'s formula. Then, the rest follows from Theorem \ref{T:1.4}.
\end{proof}
\renewcommand{\proofname}{Proof}

\bigskip

{{\bf Acknowledgement.} Part of this work was done while the second author was visiting the Department of Mathematics at the University of Washington.
The authors thank Longmin Wang and the referee for helpful comments.}

\bigskip

{{\bf Zhen-Qing Chen}

 Department of Mathematics, University of Washington, Seattle, WA 98195, USA

 E-mail: zqchen@uw.edu}

\

{{\bf Eryan Hu}

 School of Mathematics and Statistics, Beijing Institute of Technology, Beijing 100081, China

 E-mail: eryanhu@gmail.com}

\begin{thebibliography}{10}

\bibitem{Applebaum.2004}
D.~Applebaum.
\newblock {\em L\'evy Processes and Stochastic Calculus}. Cambridge University Press, Cambridge, 2004.

\bibitem{BC}  R. F. Bass and Z.-Q. Chen, Brownian motion with singular drift.
{\it Ann. Probab. \bf 31} (2003), 791-817.

\bibitem{BlumenthalGetoor.1968}
R.~M. Blumenthal and R.~K. Getoor.
\newblock {\em Markov Processes and Potential Theory}.
\newblock Pure and Applied Mathematics, Vol. 29. Academic Press, New York,
  1968.


\bibitem{BogdanJakubowski.2007}
K.~Bogdan and T.~Jakubowski.
\newblock Estimates of heat kernel of fractional {L}aplacian perturbed by gradient operators.
\newblock {\em Comm. Math. Phys., \bf  271} (2007), 179--198.

\bibitem{ChenKimSong.2011}
Z.-Q. Chen, P.~Kim, and R.~Song.
\newblock Heat kernel estimates for {$\Delta+\Delta^{\alpha/2}$} in {$C^{1,1}$}
  open sets.
\newblock {\em J. Lond. Math. Soc. (2), \bf  84} (2011),  58--80.

\bibitem{ChenKimSong.2012}
Z.-Q. Chen, P.~Kim, and R.~Song.
\newblock Dirichlet heat kernel estimates for fractional {L}aplacian with
  gradient perturbation.
\newblock {\em Ann. Probab. \bf  40} (2012), 2483--2538.

\bibitem{ChenKumagai.2003}
Z.-Q. Chen and T.~Kumagai.
\newblock Heat kernel estimates for stable-like processes on {$d$}-sets.
\newblock {\em Stochastic Process. Appl., \bf 108} (2003), 27--62.

\bibitem{ChenKumagai.2008}
Z.-Q. Chen and T.~Kumagai.
\newblock Heat kernel estimates for jump processes of mixed types on metric
  measure spaces.
\newblock {\em Probab. Theory Related Fields, \bf 140} (2008), 277--317.

\bibitem{ChenKumagai.2010}
Z.-Q. Chen and T.~Kumagai.
\newblock A priori {H}\"older estimate, parabolic {H}arnack principle and heat
  kernel estimates for diffusions with jumps.
\newblock {\em Rev. Mat. Iberoam., \bf 26} (2010), 551--589.

\bibitem{ChenWang.2012}
Z.-Q. Chen and J.~Wang.
\newblock Perturbation by non-local operators.
 	{\it Preprint}. arXiv:1312.7594 [math.PR]

\bibitem{ChenWang.2013a}
Z.-Q. {Chen} and L.~{Wang}.
\newblock {Uniqueness of stable processes with drift}.
 	{\it Preprint.}  arXiv:1309.6414 [math.PR]

\bibitem{ChenWang.2014}
Z.-Q. {Chen} and L.~{Wang}.
\newblock Heat kernel estimates for relativistic stable processes with singular
  drifts.
\newblock {\em Preprint}, 2014.

\bibitem{CranstonZhao.1987}
M.~Cranston and Z.~Zhao.
\newblock Conditional transformation of drift formula and potential theory for
  {${1\over 2}\Delta +b(\cdot)\cdot\nabla$}.
\newblock {\em Comm. Math. Phys., \bf 112} (1987), 613--625.

\bibitem{KimSong.2006}
P.~Kim and R.~Song.
\newblock Two-sided estimates on the density of {B}rownian motion with singular
  drift.
\newblock {\em Illinois J. Math., \bf 50} (2006), 635--688.

\bibitem{KS}    P. Kim and R. Song, Stable process with singular drift.
{\it Stochastic Process. Appl. \bf 124} (2014),  2479-2516.

\bibitem{SongVondravcek.2007}
R.~Song and Z.~Vondra{\v{c}}ek.
\newblock Parabolic {H}arnack inequality for the mixture of {B}rownian motion
  and stable process.
\newblock {\em Tohoku Math. J. (2), \bf 59} (2007), 1--19.

\bibitem{Zhang.1996}
Q.~Zhang.
\newblock A {H}arnack inequality for the equation {$\nabla(a\nabla u)+b\nabla
  u=0$}, when {$\vert b\vert \in K_{n+1}$}.
\newblock {\em Manuscripta Math., \bf 89} (1996), 61--77.

\bibitem{Zhang.1997}
Q.~S. Zhang.
\newblock Gaussian bounds for the fundamental solutions of {$\nabla (A\nabla
  u)+B\nabla u-u_t=0$}.
\newblock {\em Manuscripta Math., \bf  93} (1997), 381--390.

\end{thebibliography}
\end{document}